\documentclass[12pt]{amsart}
\pdfoutput=1
\usepackage[margin=1in]{geometry}
\usepackage[utf8]{inputenc}
\usepackage{graphicx, cite}
\usepackage{amsmath,amssymb}
\usepackage{amsthm}
\usepackage{color}
\usepackage{tikz}
\usetikzlibrary{arrows.meta,positioning}
\usepackage{pgfplots}
\pgfplotsset{compat=1.18}
\usetikzlibrary{fillbetween}
\allowdisplaybreaks

\makeatletter
\def\l@section{\@tocline{1}{12pt plus2pt}{0pt}{}{\bfseries}}
\def\l@subsection{\@tocline{2}{0pt}{2pc}{2pc}{}}
\makeatother
\makeatletter
\def\subsection{\@startsection{subsection}{2}{\z@}%
	{-3.25ex\@plus -1ex \@minus -.2ex}%
	{1.5ex \@plus .2ex}%
	{\normalfont\bfseries\boldmath}}
\def\subsubsection{\@startsection{subsubsection}{3}%
	\z@{.5\linespacing\@plus.7\linespacing}{-.5em}%
	{\normalfont\bfseries\boldmath}}
\renewcommand\paragraph{\@startsection{paragraph}{4}{\z@}%
	{3.25ex \@plus1ex \@minus.2ex}%
	{-1em}%
	{\normalfont\normalsize\bfseries}}
\makeatother

\theoremstyle{plain}
\newtheorem{thm}{Theorem}[section]

\newtheorem{cor}[thm]{Corollary}
\newtheorem{lem}[thm]{Lemma}
\newtheorem{prop}[thm]{Proposition}

\theoremstyle{definition}
\newtheorem{defn}[thm]{Definition}

\theoremstyle{remark}
\newtheorem{rem}[thm]{Remark}

\theoremstyle{plain}

\numberwithin{equation}{section}

\theoremstyle{plain} 
\newcommand{\thistheoremname}{}
\newtheorem{genericthm}[thm]{\thistheoremname}

  \newtheorem*{genericthm*}{\thistheoremname}
\newenvironment{namedthm*}[1]
  {\renewcommand{\thistheoremname}{#1}%
   \begin{genericthm*}}
  {\end{genericthm*}}

\newcommand{\D}{{\mathbb D}}

\newcommand{\R}{{\mathbb R}}

\newcommand{\C}{{\mathbb C}}

\newcommand{\Z}{{\mathbb Z}}
\newcommand{\calR}{{\mathcal R}}

\newcommand{\calD}{{\mathcal D}}
\newcommand{\calM}{{\mathcal M}}

\newcommand{\calH}{{\mathcal H}}
\newcommand{\calC}{{\mathcal C}}

\newcommand{\calS}{{\mathcal S}}
\newcommand{\supp}{{\textnormal{supp}}}

\makeatletter
\newcommand{\vast}{\bBigg@{4}}
\newcommand{\Vast}{\bBigg@{5}}
\makeatother

\def\udot#1{\ifmmode\oalign{$#1$\crcr\hidewidth.\hidewidth
    }\else\oalign{#1\crcr\hidewidth.\hidewidth}\fi}

\def\R{\mathbb{R}}
\def\Z{\mathbb{Z}}
\def\T{\mathbb{T}}
\def\C{\mathbb{C}}

\def\beq{\begin{equation}}
\def\eeq{\end{equation}}
\makeatletter
\newcommand{\doublewidetilde}[1]{{%
  \mathpalette\double@widetilde{#1}%
}}
\newcommand{\double@widetilde}[2]{%
  \sbox\z@{$\m@th#1\widetilde{#2}$}%
  \ht\z@=.9\ht\z@
  \widetilde{\box\z@}%
}
\makeatother

\def\one{\mbox{1\hspace{-4.25pt}\fontsize{12}{14.4}\selectfont\textrm{1}}}

\makeatletter
\def\@makefnmark{%
  \leavevmode
  \raise.9ex\hbox{\fontsize\sf@size\z@\normalfont\tiny\@thefnmark}}
\makeatother

\begin{document}
	
\title[]{A Full Characterization of the Dirichlet Carleson embedding $\operatorname{id}: \calD_{p-1}^p \to L^p(\mu)$ for $p>2$}

\author{Bingyang Hu}
\address{(Bingyang Hu) Department of Mathematics and Statistics\\
        Auburn University\\
        Auburn, Alabama, USA, 36849}
\email{bzh0108@auburn.edu}

\author{Xiaojing Zhou}
\address{(Xiaojing Zhou) Department of Mathematics and Statistics\\
         Auburn University\\
         Auburn, Alabama, USA, 36849}
\email{xiz0003@auburn.edu}

 \begin{abstract}
In this paper, we obtain a full characterization of the finite positive Borel measures $\mu$ on $\D$ for which the embedding
$$
\operatorname{id}:\calD_{p-1}^p\longrightarrow L^p(\mu),\qquad p>2,
$$
is bounded. More precisely, for any dyadic system $\calD$ on $\T$, we prove that this embedding is bounded if and only if
$$
\calC_{p,\calD}(\mu)+\calH_{p,\calD}(\mu)<\infty,
$$
where $\calC_{p,\calD}(\mu)$ and $\calH_{p,\calD}(\mu)$ denote the packing energy and the Haar energy of $\mu$, respectively. This resolves a longstanding characterization problem in the theory of Dirichlet-type spaces that arose from Wu's 1999 conjecture, corresponds to the endpoint case not covered by the work of Arcozzi, Rochberg, and Sawyer in 2002, and remained open after the works of Girela and Pel\'aez in 2006 and Galanopoulos, Girela, and Pel\'aez in 2011. We also construct finite measures showing that the two energy conditions are genuinely distinct. The key ingredient in the proof is a reduction of the Dirichlet embedding to a dyadic Whitney embedding, which allows us to combine weighted Hardy inequalities on trees with probabilistic arguments.
\end{abstract}

\date{\today}


\maketitle
\tableofcontents

\section{Introduction}

Let $\D$ denote the unit disc in $\C$, $\T:=\partial\D$ be the unit circle, $dA$ be the normalized area measure on $\D$ and $H(\D)$ be the space of holomorphic functions on $\D$, endowed with the compact--open topology. For $0<p<\infty$ and $\alpha>-1$, the weighted Dirichlet-type space $\calD_\alpha^p$ consists of all $f \in H(\D)$ such that 
$$
\|f\|_{\calD_\alpha^p}^p
:=|f(0)|^p+\int_\D|f'(z)|^p(1-|z|^2)^\alpha\,dA(z)<\infty.
$$
The main \emph{goal} of this paper is to completely characterize the positive Borel measures $\mu$ for which
\begin{equation}\label{Dirichletembedding}
\operatorname{id}: \calD_{p-1}^p \longrightarrow L^p(\mu)
\end{equation}
is bounded for $2<p<\infty$, thereby resolving the longstanding characterization problem that remained open following the works of Wu \cite{Wu1999}, Girela--Pel\'aez \cite{GirelaPelaez2006}, and Galanopoulos--Girela--Pel\'aez \cite{GGP2011}. In particular, \eqref{Dirichletembedding} is precisely the endpoint case \emph{not} covered by the characterization of Arcozzi, Rochberg, and Sawyer \cite{ArcozziRochbergSawyer2002}; see \eqref{20260830eq01R}. 

\vspace{0.1cm}

The characterization of Carleson measures has long been a \emph{central} theme in complex function theory. We briefly review below the developments most relevant to the problem \eqref{Dirichletembedding}.

\begin{enumerate}
\item To the best of our knowledge, the systematic study of this problem dates back to Carleson's seminal work \cite{Carleson1962}, in which he proved that the Hardy Carleson embedding 
$$
H^p\hookrightarrow L^p(\mu)
$$
is bounded if and only if 
\begin{equation} \label{20260829cond01}
\mu(Q_I)\lesssim |I|,\qquad I\subseteq\T,
\end{equation} 
where $H^p$ denotes the classical Hardy space and $Q_I$ denotes the Carleson tent associated with $I$; see \eqref{20260829defn01}. 

\vspace{0.1cm}

\item Hastings \cite{Hastings1975} subsequently proved that the Bergman embedding 
$$
A^p\hookrightarrow L^p(\mu)
$$ 
is bounded if and only if
\begin{equation} \label{20260829cond02}
\mu(Q_I)\lesssim |I|^2,\qquad I\subseteq\T,
\end{equation} 
where $A^p$ denotes the classical Bergman space.

\vspace{0.1cm}

\item The situation changes substantially for the classical Dirichlet space $\calD_0^2$. Stegenga \cite{Stegenga1980} characterized the positive measures $\mu$ for which
$$
\operatorname{id}:\calD_0^2\longrightarrow L^2(\mu)
$$
is bounded, in terms of the logarithmic capacity of subsets of $\T$. In particular, unlike the Hardy and Bergman cases, conditions of the form \eqref{20260829cond01} or \eqref{20260829cond02}, do not provide a necessary-and-sufficient characterization. Verbitsky \cite{Verbitsky1985} subsequently extended this capacity approach to the analytic Besov spaces
$\calD_{p-2}^p, \ 1<p<\infty$
and characterized the boundedness of
$$
\operatorname{id}:\calD_{p-2}^p \longrightarrow L^p(\mu)
$$
in terms of the corresponding nonlinear capacities. Kerman and Sawyer later extended this line of research to certain radially weighted Dirichlet spaces \cite{KermanSawyer1988}.
\end{enumerate}

 We now turn to the developments concerning the Dirichlet Carleson embedding \eqref{Dirichletembedding}, which is the main subject of the present paper. 

 \vspace{0.1cm}

\begin{enumerate}

\item [(4)] Wu \cite{Wu1999} proved that the embedding \eqref{Dirichletembedding} is bounded if and only if the condition \eqref{20260829cond01} holds for $0<p\leq2$. He further conjectured that the same characterization should remain valid for $p>2$.

\item[(5)] Arcozzi, Rochberg, and Sawyer \cite{ArcozziRochbergSawyer2002} developed an elegant discrete tree characterization of the positive measures $\mu$ for which
\begin{equation} \label{20260830eq01R}
\operatorname{id}:\calD_\alpha^p\longrightarrow L^p(\mu)
\end{equation} 
is bounded in the range
$$
1<p<\infty,\qquad -1<\alpha<p-1.
$$
The embedding \eqref{Dirichletembedding} considered in the present paper corresponds precisely to the endpoint $\alpha=p-1$, which lies outside the range covered by their result.

\vspace{0.1cm}

\item [(6)] In the critical range $p>2$, Girela and Pel\'aez \cite{GirelaPelaez2006} constructed a positive measure $\mu$ satisfying \eqref{20260829cond01}
for which the embedding \eqref{Dirichletembedding}
fails to be bounded, thereby disproving Wu's conjecture. Galanopoulos, Girela, and Pel\'aez \cite{GGP2011} subsequently showed that the Arcozzi--Rochberg--Sawyer condition remains sufficient when $p>2$ but is no longer necessary. More recently, Pel\'aez, P\'erez-Gonz\'alez, and R\"atty\"a \cite{PPR2014} obtained the sharp logarithmic one-box sufficient condition
$$
\mu(Q_I)\lesssim |I|\left(\log\frac{e}{|I|}\right)^{1-\frac{p}{2}} \qquad \textrm{for} \qquad 2<p<+\infty.
$$
\end{enumerate}

\medskip 

The preceding work completely characterizes \eqref{Dirichletembedding} for $0<p\leq2$, but no characterization was known for $2<p<\infty$. Thus, $p>2$ is the only remaining case on the critical scale $\calD_{p-1}^p$. Understanding this remaining case is important not only for completing the Carleson measure theory of the critical Dirichlet scale, but also for the study of several related problems in which Carleson embedding methods for Dirichlet-type spaces play an essential role:
\begin{enumerate}
\item[(a)] embeddings of $Q_p$ and $F(p,q,s)$ spaces into tent spaces \cite{Xiao2008,PauZhao2014,LiuLouZhu2017,HuZhou2025};
\item[(b)] boundedness and compactness of area operators \cite{WuArea2006,GongLouWu2010,Wu2011};
\item[(c)] Volterra-type and integration operators, as well as multipliers, on spaces of Dirichlet type \cite{GirelaPelaezJFA2006,GGP2011,Wu2011, HuZhou2025}.
\end{enumerate}

\medskip 

Returning now to our solution to the problem \eqref{Dirichletembedding}, we first observe that it suffices to consider the case $\mu(\D)<+\infty$, since otherwise the embedding \eqref{Dirichletembedding} fails to be bounded. Let $\calD$ be a dyadic system on $\T$. For each $I\in\calD$, let $I_{+}$ and $I_{-}$ denote its two dyadic children, ordered counterclockwise, and let
\begin{equation} \label{20260829defn01}
Q_I:=\left\{z=re^{2\pi i t} \in\D:1-|I| \le r<1,\ e^{2\pi i t} \in I\right\}
\end{equation} 
be the \emph{Carleson tent} associated with $I$.

To state our main result, we introduce two dyadic energies associated with $\mu$.

\begin{defn} \label{20260825defn01}
Let $p>2$ and let $\mu$ be a finite positive Borel measure on $\D$. We define the \emph{packing energy} of $\mu$ by
$$
\calC_{p,\calD}(\mu):=
\sup_{\substack{J\in\calD\\ \mu(Q_J)>0}}
\frac{1}{\mu(Q_J)}
\left[ \sum_{\substack{I\in\calD\\ I\subseteq J}}
\left(\frac{\mu(Q_I)}{|I|}\right)^{\frac{p}{p-2}}|I| \right],
$$
and the \emph{Haar energy} of $\mu$ by
$$
\calH_{p,\calD}(\mu):=
\sup_{\substack{J\in\calD\\ \mu(Q_J)>0}}
\frac{1}{\mu(Q_J)}
\left[\sum_{\substack{I\in\calD\\ I\subseteq J}}
\left(\frac{\left|\mu(Q_{I_+})-\mu(Q_{I_-})\right|}{|I|}\right)^{\frac{p}{p-1}}|I| \right].
$$
Finally, if $\mu=0$, then set $\calC_{p, \calD}(\mu)=\calH_{p, \calD}(\mu)=0$. 
\end{defn}

Our first main result is the following, which gives an answer to the problem \eqref{Dirichletembedding}. 

\begin{thm}\label{20260813thmmain}
Let $2<p<\infty$, $\calD$ be any dyadic system on $\T$, and $\mu$ be a finite positive Borel measure on $\D$. Then
\begin{equation} \label{Dirichletembedding01}
\int_\D|f(z)|^p\,d\mu(z)
\leq C\|f\|_{\calD_{p-1}^p}^p,
\qquad f\in\calD_{p-1}^p
\end{equation} 
if and only if
$$
\calC_{p,\calD}(\mu)+\calH_{p,\calD}(\mu)<\infty.
$$
Moreover, if $C_{best}$ denotes the optimal constant in \eqref{Dirichletembedding01}, then
$$
C_{best} \simeq_p \calC_{p,\calD}(\mu)^{\frac{p-2}{2}}+ \calH_{p,\calD}(\mu)^{p-1}.
$$
\end{thm}

The proof of Theorem~\ref{20260813thmmain} is motivated by the tree approach of Arcozzi, Rochberg, and Sawyer \cite{ArcozziRochbergSawyer2002} and by ideas from recent developments in dyadic analysis on Carleson tents, particularly those in \cite{HuZhou2026,HuLuoXiaoZhou2026,HuXiaoZhou2026}. The main idea is to replace the harmonic embedding with a dyadic Whitney embedding, where tools from dyadic harmonic analysis can be applied more directly. Since the proof involves several reductions, we summarize its structure below.

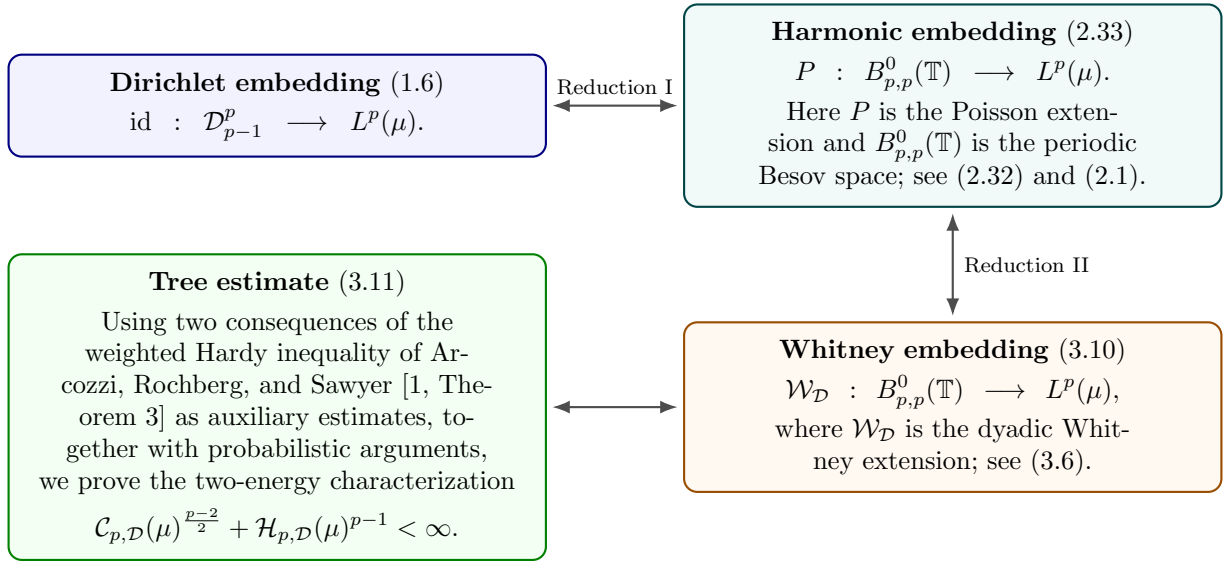
\vspace{0.1cm}

\begin{center}
\begin{tikzpicture}[
    font=\footnotesize,
    node distance=15mm and 18mm,
    box/.style={
        rectangle,
        rounded corners=5pt,
        thick,
        align=center,
        text width=0.405\textwidth,
        inner sep=6pt
    },
    arrow/.style={
        <->,
        >=Latex,
        thick,
        draw=black!70,
        shorten <=2pt,
        shorten >=2pt
    }
]
\node[box,draw=blue!50!black,fill=blue!5] (dirichlet) {
\textbf{Dirichlet embedding} \eqref{Dirichletembedding01}\\[1mm]
$\displaystyle
\operatorname{id}:\calD_{p-1}^p\longrightarrow L^p(\mu).
$
};
\node[box,draw=teal!55!black,fill=teal!5,right=of dirichlet] (harmonic) {
\textbf{Harmonic embedding} \eqref{20260819eq01}\\[1mm]
$\displaystyle
P:B_{p,p}^0(\T)\longrightarrow L^p(\mu).
$\\[1mm]
Here $P$ is the Poisson extension and $B_{p,p}^0(\T)$ is the periodic Besov space; see \eqref{20260819eq00} and \eqref{20260829defn01X}.
};
\node[box,draw=orange!60!black,fill=orange!5,below=of harmonic] (whitney) {
\textbf{Whitney embedding} \eqref{20260820eq08}\\[1mm]
$\displaystyle
\mathcal W_{\calD}:B_{p,p}^0(\T)\longrightarrow L^p(\mu),
$\\[1mm]
where $\mathcal W_{\calD}$ is the dyadic Whitney extension; see \eqref{20260820eq06}.
};
\node[box,draw=green!50!black,fill=green!5,left=of whitney] (tree) {
\textbf{Tree estimate} \eqref{20260820eq09}\\[1mm]
Using two consequences of the weighted Hardy inequality of Arcozzi, Rochberg, and Sawyer
\cite[Theorem~3]{ArcozziRochbergSawyer2002} as auxiliary estimates, together with probabilistic arguments, we prove the two-energy characterization
\[
\calC_{p,\calD}(\mu)^{\frac{p-2}{2}}
+\calH_{p,\calD}(\mu)^{p-1}<\infty.
\]
};
\draw[arrow] 
(dirichlet.east) -- node[midway,above] {\tiny Reduction I} (harmonic.west);

\draw[arrow] 
(harmonic.south) -- node[midway,right] {\tiny Reduction II} (whitney.north);
\draw[arrow] (whitney.west) -- (tree.east);
\end{tikzpicture}
\end{center}

\begin{rem}
The main difference between our approach and that of Arcozzi, Rochberg, and Sawyer \cite{ArcozziRochbergSawyer2002} lies in the role of cancellation. In the regime $-1<\alpha<p-1$, their discretization reduces the embedding problem to a positive Hardy inequality on a tree (or equivalently, the collection of all upper Carleson tents), with no cancellation involved; see \cite[Theorem 1]{ArcozziRochbergSawyer2002}.

\vspace{0.1cm}

At the endpoint $\alpha=p-1$, however, this positive tree model no longer yields an exact characterization, as shown by Galanopoulos, Girela, and Pel\'aez \cite{GGP2011}. Our proof instead passes through the boundary Besov space and the dyadic Whitney extension, leading to \emph{signed sums} of Haar coefficients along dyadic chains of upper Carleson tents; see \eqref{20260820eq09}. The packing energy $\calC_{p, \calD}(\mu)$ controls the martingale component of these sums, whereas the Haar energy measures the imbalance between the two dyadic children. Thus, although two auxiliary estimates, Corollaries \ref{20260825cor01} and \ref{20260825cor02} in our proof, follow from the general weighted Hardy inequality \cite[Theorem~3]{ArcozziRochbergSawyer2002} of Arcozzi, Rochberg, and Sawyer, the endpoint reduction and the resulting two-energy analysis are designed to capture the cancellation of the Haar coefficients, a feature absent from the Arcozzi--Rochberg--Sawyer framework.
\end{rem}

We also show that the Haar energy $\calH_{p, \calD}(\mu)$ cannot be controlled by the
packing energy $\calC_{p, \calD}(\mu^{(N)})$. More precisely, we have the following result. 

\begin{thm} \label{mainthm02}
Let $2<p<\infty$ and let $\calD$ be a dyadic system on $\T$. For every integer $N \ge 2$, there exists a finite measure $\mu^{(N)}$ on $\D$, such that 
$$
\calC_{p, \calD} \left(\mu^{(N)} \right) \lesssim_p 1, \qquad \textrm{and} \qquad \calH_{p, \calD} \left( \mu^{(N)} \right) \gtrsim_p N^{\frac{p}{p-1}}. 
$$
\end{thm}

The rest of the paper is organized as follows. In Section~\ref{Sec02}, we identify $\calD_{p-1}^p$ with the analytic subspace of the periodic Besov space $B_{p,p}^0(\T)$ and reduce the Dirichlet embedding to a harmonic embedding. In Section~\ref{Sec03}, we introduce the dyadic Whitney extension, establish the equivalence between the harmonic and Whitney embeddings, and reformulate the latter as an estimate on the dyadic tree. Section~\ref{Sec04} recalls the weighted Hardy inequality of Arcozzi, Rochberg, and Sawyer and derives the two consequences associated with the packing and Haar energies. In Section~\ref{Sec05}, we combine these reductions with the weighted Hardy estimates and martingale arguments to prove Theorem~\ref{20260813thmmain}. Finally, in Section~\ref{Sec06}, we construct finite atomic measures with uniformly bounded packing energy and arbitrarily large Haar energy, thereby proving Theorem~\ref{mainthm02}.

Throughout this paper, for any two nonnegative quantities $a$ and $b$, we write $a \lesssim b$ if there exists an absolute constant $C > 0$ independent of $a$ and $b$ such that $a \le Cb$. The notation $a \simeq b$ indicates that both $a \lesssim b$ and $b \lesssim a$ hold simultaneously. Finally, if the implicit constants may depend on $p$, we write $\lesssim_p$ and $\simeq_p$.
\bigskip

\noindent{\textbf{Acknowledgments.}}
The authors would like to thank Ruhan Zhao and Jouni R\"atty\"a for bringing the problem studied in this paper to their attention. They are also grateful to Jos\'e \'Angel Pel\'aez for helpful discussions and for informing them of related work in progress. The first author was supported by NSF grant DMS-2555999 and by the Simons Foundation Travel Support for Mathematicians program under award MPS-TSM-00007213.

\bigskip 

\section{Reduction I: from Dirichlet embedding to harmonic embedding} \label{Sec02}

We first recall some basic facts about periodic function spaces and the Littlewood--Paley characterization of the Dirichlet space $\calD_{p-1}^p$, which may be viewed as periodic counterparts of the corresponding results on $\R^n$. For the notation and general setup of periodic function spaces, we follow Schmeisser and Triebel \cite[Section~3]{SchmeisserTriebel1987}.

\vspace{0.1cm}

To begin with, we identify $\T$ with $\R/\Z$, equipped with normalized Lebesgue measure, and let $\mathcal D'(\T)$ denote the space of distributions on $\T$. For $h\in\mathcal D'(\T)$, its \emph{Fourier coefficients} are defined by
$$
\widehat h(n):=\left\langle h,e^{-2\pi in(\cdot)}\right\rangle,
\qquad n\in\Z.
$$
In particular, if $h\in L^1(\T)$, then
$$
\widehat h(n)=\int_0^1h(t)e^{-2\pi int}\,dt,
\qquad n\in\Z.
$$

Fix an even function $\chi\in C_c^\infty(\R)$ such that $\chi$ is nonincreasing on $[0,\infty)$ and
$$
0\leq\chi\leq1,\qquad
\chi(\xi)=1\quad\text{for }|\xi|\leq1,
\qquad
\chi(\xi)=0\quad\text{for }|\xi|\geq2.
$$
Set
$$
\varphi_0(\xi):=\chi(\xi)
$$
and, for $j\geq1$,
$$
\varphi_j(\xi):=\chi(2^{-j}\xi)-\chi(2^{-j+1}\xi).
$$
Then $\varphi_j\geq0$ and
$$
\sum_{j=0}^\infty\varphi_j(\xi)=1,
\qquad \xi\in\R.
$$

For $h\in\mathcal D'(\T)$ and $j\geq0$, define the \emph{Littlewood--Paley projection}
$$
(\Delta_jh)(t):=
\sum_{n\in\Z}\varphi_j(n)\widehat h(n)e^{2\pi int},
\qquad t\in\T.
$$
For $1<p<\infty$, the \emph{periodic Besov space} $B_{p,p}^0(\T)$ consists of all $h\in\mathcal D'(\T)$ such that
\begin{equation} \label{20260829defn01X}
\|h\|_{B_{p,p}^0(\T)}^p:=\sum_{j=0}^\infty
\|\Delta_jh\|_{L^p(\T)}^p<+\infty.
\end{equation} 
We further denote by
$$
B_{p,p,+}^0(\T):=\left\{h\in B_{p,p}^0(\T):
\widehat h(n)=0\ \text{for every }n<0
\right\}
$$
the analytic subspace of $B_{p,p}^0(\T)$.

The following result identifies the Dirichlet-type space $\calD_{p-1}^p$ with the analytic periodic Besov space $B_{p,p,+}^0(\T)$. We believe that this identification is implicitly contained in the literature, and its counterpart on $\R^n$ is routine. Since the precise periodic analytic formulation needed here does not seem to be readily available, we include a proof for the reader's convenience.

\begin{prop}\label{20260813eq01}
Let $1<p<\infty$. For every analytic polynomial
\begin{equation} \label{20260814eq03}
f(z)=\sum_{n=0}^N a_nz^n,
\end{equation}
define
\begin{equation} \label{20260819eq02A}
f^*(t):=\sum_{n=0}^N a_ne^{2\pi int},\qquad t\in\T.
\end{equation} 
Then
\begin{equation}\label{20260818eq01}
\|f\|_{\calD_{p-1}^p}^p\simeq_p\|f^*\|_{B_{p,p}^0(\T)}^p,
\end{equation}
where the implicit constants depend only on $p$ and are independent of $N$ and the coefficients $\{a_n\}_{n=0}^N$. 

Consequently, the map $f\mapsto f^*$, initially defined on analytic polynomials, extends uniquely to an isomorphism
\begin{equation} \label{20260818eq02}
\operatorname{Tr}:\calD_{p-1}^p\longrightarrow B_{p,p,+}^0(\T)
\end{equation} 
with equivalent norms. More precisely, if $f(z)=\sum_{n=0}^\infty a_nz^n\in\calD_{p-1}^p$, then the distribution $\operatorname{Tr}f$ is characterized by
$$
\widehat{\operatorname{Tr}f}(n)=
\begin{cases}
a_n,&n\geq0,\\
\\
0,&n<0.
\end{cases}
$$
\end{prop}

\begin{proof}

We divide the proof into several steps. 

\vspace{0.1cm}

\noindent\textbf{Step I.} Let $f$ be an analytic polynomial of the form \eqref{20260814eq03}. We first prove that
\begin{equation}\label{20260816eq45}
|f(0)|^p+\int_\D|f'(z)|^p(1-|z|^2)^{p-1}\,dA(z)\lesssim\sum_{j=0}^\infty\|\Delta_jf^*\|_{L^p(\T)}^p.
\end{equation}
To begin with, since $\varphi_0(0)=\chi(0)=1$, we have 
$$
|f(0)|=\left| \widehat{\Delta_0 f^*}(0) \right|=\left|\int_0^1 \Delta_0 f^*(t)dt \right| \le \left\|\Delta_0 f^* \right\|_{L^p(\T)},
$$
which implies 
$$
|f(0)|^p \lesssim \textrm{RHS of \eqref{20260816eq45}}. 
$$
Therefore, it suffices to prove 
\begin{equation} \label{20260816eq46}
\int_\D|f'(z)|^p(1-|z|^2)^{p-1}\,dA(z)\lesssim\sum_{j=0}^\infty\|\Delta_jf^*\|_{L^p(\T)}^p.
\end{equation} 
We have the following claim: for any $j \ge 0$ and $1/2 \le r<1$, one has
\begin{equation}\label{20260814claim01}
\|R(\Delta_jf)(re^{2\pi i(\cdot)})\|_{L^p(\T)}\lesssim2^je^{-c2^j(1-r)}\|\Delta_jf^*\|_{L^p(\T)}, 
\end{equation}
where $c>0$ is an absolute constant, $Rf(z):=zf'(z)$ denotes the radial derivative of $f$, and
\begin{equation} \label{20260814eq01}
\Delta_jf(z):=\sum_{k\geq0}\varphi_j(k)a_kz^k \qquad \textrm{for} \qquad f(z)=\sum_{k \ge 0} a_kz^k. 
\end{equation}
In particular, $(\Delta_jf)^*=\Delta_jf^*$ for every analytic polynomial $f$. 

\vspace{0.1cm}
We first prove \eqref{20260814claim01} for the case when $j=0$. Observe that the support condition on $\varphi_0$ gives
\begin{equation} \label{20260817eq05}
\Delta_0f^*(t)=a_0+a_1e^{2\pi it}
\quad\text{and}\quad
R(\Delta_0f)(re^{2\pi it})=ra_1e^{2\pi it}.
\end{equation} 
Therefore,
$$
\|R(\Delta_0f)(re^{2\pi i(\cdot)})\|_{L^p(\T)} \simeq |a_1|=|\widehat{\Delta_0f^*}(1)|\leq\|\Delta_0f^*\|_{L^p(\T)}.
$$
Since $e^{-c(1-r)}$ is bounded below for $1/2\leq r<1$, \eqref{20260814claim01} also holds for $j=0$.

\vspace{0.1cm}

Next, we show \eqref{20260814claim01} for $j\geq1$. In this case, recall that the Fourier support of $\Delta_jf^*$ is contained in $\{n\in\Z:2^{j-1}\le n \le 2^{j+1}\}$. Choose a smooth function\footnote{For example, one may take $\psi_j(\xi):=\psi(2^{-j}\xi)$, where $\psi\in C_c^\infty(\R)$ satisfies $\supp \ \psi \subseteq [1/4,4]$, $0 \le \psi \le 1$, and $\psi=1$ on $[1/2,2]$.} $\psi_j \in C_c^\infty(\R)$ such that
\begin{enumerate}
\item [$\bullet$] $\supp \ \psi_j \subseteq \{\xi\in\R:2^{j-2}\le \xi \le 2^{j+2}\}$;
    \item [$\bullet$] $\psi_j(\xi)=1, \ \xi \in \supp \ \varphi_j \cap (0, \infty)$;
    \item [$\bullet$] $|\psi_j'(\xi)| \lesssim 2^{-j},\;  \xi \in \R$.
\end{enumerate}
By \eqref{20260814eq01}, 
\begin{align*}
R(\Delta_jf)(re^{2\pi it})
&=\sum_{k\ge 0}kr^k\varphi_j(k)a_ke^{2\pi ikt} \\
&=\sum_{k \ge 0} kr^k \psi_j(k) \cdot  \varphi_j(k)a_ke^{2\pi ikt} \\
&=\sum_{k \ge 0} m_{j, r}(k) \cdot  \varphi_j(k)a_ke^{2\pi ikt} \\
&= T_{m_{j, r}}(\Delta_j f^*), 
\end{align*} 
where $T_{m_{j,r}}$ denotes the periodic Fourier multiplier with symbol
$$
m_{j,r}(k):=
\begin{cases}
kr^k\psi_j(k),\qquad k\geq0; \\
\\
0, \qquad \hfill k<0. 
\end{cases} 
$$
We next verify the two hypotheses of the strong Marcinkiewicz multiplier theorem for $T_{m_{j, r}}$.

\vspace{0.1cm}

\noindent \textit{\underline{Claim I}: 
\begin{equation} \label{20260815eq01}
|m_{j, r}(k)| \lesssim 2^j e^{-c2^j(1-r)}
\end{equation} 
for some absolute constant $c>0$.}

\vspace{0.1cm}

Note that it suffices to consider the case when $m_{j,r}(k) \neq 0$, then one has $k\simeq2^j$. Since $\log r\leq-(1-r)$, it follows that
\begin{align*}
|m_{j,r}(k)|=kr^k|\psi_j(k)|\lesssim2^je^{-k(1-r)} \lesssim2^je^{-c2^j(1-r)},
\end{align*}
which gives \eqref{20260815eq01}. 

\vspace{0.1cm}

\noindent \textit{\underline{Claim II}: 
\begin{equation} \label{20260815eq02}
\sum_{k\in\Z}|m_{j,r}(k+1)-m_{j,r}(k)| \lesssim 2^j e^{-c2^j(1-r)}
\end{equation}
for some absolute constant $c>0$.}

\vspace{0.1cm}

It suffices to consider the case where $m_{j,r}(k+1)-m_{j,r}(k)\neq0$. Since $m_{j,r}(k)=0$ for every $k<0$, we may further restrict ourselves to $k\geq0$. The case $k=0$ can be handled directly. Indeed, if
$$
|m_{j,r}(1)-m_{j,r}(0)|=r|\psi_j(1)|\neq0,
$$
then the support condition on $\psi_j$ implies that $j=1$ or $j=2$. Since $\frac12\leq r<1$, it follows in either case that
$$
|m_{j,r}(1)-m_{j,r}(0)|\lesssim e^{-c2^j(1-r)}.
$$
Thus, it remains to consider $k\geq1$, and we consider two cases. 

\vspace{0.1cm}

\noindent $\bullet$ If both $m_{j, r}(k+1), \; m_{j, r}(k) \neq 0$, then 
\begin{align} \label{20260815eq10}
m_{j,r}(k+1)-m_{j,r}(k)
&=(k+1)r^{k+1}\psi_j(k+1)-kr^k\psi_j(k) \nonumber \\
&=\left( \left(k+1 \right)r^{k+1}-kr^k \right) \psi_j(k+1)+kr^k \left(\psi_j(k+1)-\psi_j(k) \right). 
\end{align}
Using the facts that $k \simeq 2^j$, $|\psi_j| \lesssim 1$, and $|\psi'_j| \lesssim 2^{-j}$,  \eqref{20260815eq10} gives 
$$
|m_{j,r}(k+1)-m_{j,r}(k)|\lesssim e^{-c2^j(1-r)}\big(1+2^j(1-r)\big)
$$
for some absolute constant $c>0$. Since there are at most $O(2^j)$ many $k$'s for which the difference $m_{j,r}(k+1)-m_{j,r}(k)$ is nonzero, one has 
\begin{align*}
\sum_{k\in\Z}|m_{j,r}(k+1)-m_{j,r}(k)|
&\lesssim2^je^{-c2^j(1-r)}\big(1+2^j(1-r)\big)\\
&\lesssim2^je^{-c'2^j(1-r)}
\end{align*}
for some absolute constant $c'>0$. This proves \eqref{20260815eq02}.

\vspace{0.1cm}

\noindent $\bullet$ If $m_{j,r}(k+1)=0$, then $\psi_j(k+1)=0$. Hence
\begin{align}
|m_{j,r}(k+1)-m_{j,r}(k)|
&=|m_{j,r}(k)|=kr^k|\psi_j(k)|\nonumber\\
&=kr^k|\psi_j(k+1)-\psi_j(k)|\nonumber\\
&\lesssim kr^k2^{-j}\simeq r^k\nonumber\\
&\lesssim e^{-c2^j(1-r)}
\label{20260815eq20}
\end{align}
for some absolute constant $c>0$. Similarly, if $m_{j,r}(k)=0$, the same argument yields \eqref{20260815eq20}. Finally, for each fixed $j$, there are only $O(1)$ integers $k$ for which exactly one of $m_{j,r}(k)$ and $m_{j,r}(k+1)$ vanishes. Therefore, summing \eqref{20260815eq20} over all such $k$ proves \eqref{20260815eq02} in this case.

\vspace{0.1cm}

The proof of \emph{Claim II} is complete. 

\vspace{0.15cm}

Consequently, \eqref{20260815eq01} and \eqref{20260815eq02}, together with the strong Marcinkiewicz multiplier theorem \cite[Theorem~8.2.1]{EdwardsGaudry1977} give
$$
\|T_{m_{j,r}}h\|_{L^p(\T)}\lesssim_p2^je^{-c2^j(1-r)}\|h\|_{L^p(\T)},
$$
for some absolute constant $c>0$. Finally, applying this estimate with $h=\Delta_jf^*$ proves \eqref{20260814claim01} for every $j\geq1$.

\medskip

Having established \eqref{20260814claim01}, we now return to the proof of the main estimate \eqref{20260816eq46}. For $0\leq r<1$, write
$$
M_p(f',r):=\left(\int_0^1|f'(re^{2\pi it})|^p\,dt\right)^{1/p}.
$$
Hence, using polar coordinates, we have
\begin{align} \label{20260816eq43B}
\textrm{LHS of \eqref{20260816eq46}}
&=2\int_0^1M_p(f',r)^p(1-r^2)^{p-1}r\,dr \nonumber \\
&\lesssim\int_0^{1/2}M_p(f',r)^p\,dr+\int_{1/2}^1M_p(f',r)^p(1-r)^{p-1}\,dr \nonumber \\
&:=\mathfrak L_1+\mathfrak L_2,
\end{align}
where 
$$
\mathfrak L_1:=\int_0^{1/2}M_p(f',r)^p\,dr
$$
and 
$$
\mathfrak L_2:=\int_{1/2}^1M_p(f',r)^p(1-r)^{p-1}\,dr.
$$

\medskip 

\noindent \textbf{Estimate of $\mathfrak L_1$.}  Since $|f'|^p$ is subharmonic, the function $r\mapsto M_p(f',r)$ is nondecreasing. Therefore, for every $0\leq r\leq1/2$,
\begin{equation} \label{20260816eq01}
M_p(f',r) \leq M_p\left(f',\frac12\right)
=2\left\|Rf\left(\frac12e^{2\pi i(\cdot)}\right)\right\|_{L^p(\T)}.
\end{equation}
Since $\sum_{j\geq0}\varphi_j(k)=1$ for every $k\in\Z$, we have
\begin{equation} \label{20260816eq02}
f=\sum_{j=0}^\infty\Delta_jf
\qquad\text{and}\qquad
Rf=\sum_{j=0}^\infty R(\Delta_jf).
\end{equation} 
Note that the above sums are finite because $f$ is a polynomial. Hence, by \eqref{20260814claim01}, \eqref{20260816eq01}, \eqref{20260816eq02}, and H\"older
\begin{align*}
M_p(f',r)^p
& \lesssim \left\|Rf\left(\frac12e^{2\pi i(\cdot)}\right)\right\|^p_{L^p(\T)} \\
&\lesssim \left(\sum_{j=0}^\infty\left\|R(\Delta_jf)\left(\frac12e^{2\pi i(\cdot)}\right)\right\|_{L^p(\T)} \right)^p \\
&\lesssim \left(\sum_{j=0}^\infty2^je^{-c2^j}\|\Delta_jf^*\|_{L^p(\T)} \right)^p \\
&\lesssim \sum_{j=0}^\infty\|\Delta_jf^*\|_{L^p(\T)}^p,
\end{align*}
where in the last estimate above, we have used the fact that $\{2^j e^{-c2^j}\}_{j \ge 0} \in \ell^{\frac{p}{p-1}}$. Moreover, the implicit constant above is independent of $r\in[0,1/2]$. Integrating the above estimate over this interval, we obtain
\begin{equation}\label{20260816eq43A}
\mathfrak L_1\lesssim\sum_{j=0}^\infty\|\Delta_jf^*\|_{L^p(\T)}^p.
\end{equation}

\medskip

\noindent \textbf{Estimate of $\mathfrak L_2$.} Since $r\geq1/2$, we have
$$
|f'(re^{2\pi it})|=r^{-1}|Rf(re^{2\pi it})|\leq2|Rf(re^{2\pi it})|.
$$
Using the decomposition $Rf=\sum_{j\geq0}R(\Delta_jf)$ and \eqref{20260814claim01}, we obtain, for any $r\in A_n:=[1-2^{-n}, 1-2^{-n-1})$,
\begin{align*}
M_p(f',r)
& \lesssim \|Rf(re^{2\pi i(\cdot)})\|_{L^p(\T)} \\
&\lesssim\sum_{j=0}^\infty\|R(\Delta_jf)(re^{2\pi i(\cdot)})\|_{L^p(\T)}\\
&\lesssim\sum_{j=0}^\infty2^je^{-c2^j(1-r)}\|\Delta_jf^*\|_{L^p(\T)}\\
&\lesssim\sum_{j=0}^\infty2^je^{-c2^{j-n}}\|\Delta_jf^*\|_{L^p(\T)}.
\end{align*}
Since $|A_n|\simeq2^{-n}$ and $(1-r)^{p-1}\simeq2^{-n(p-1)}$ on $A_n$, it follows that
\begin{align}
\int_{A_n}M_p(f',r)^p(1-r)^{p-1}\,dr
&\lesssim2^{-np}\left(\sum_{j=0}^\infty2^je^{-c2^{j-n}}\|\Delta_jf^*\|_{L^p(\T)}\right)^p\nonumber\\
&=\left(\sum_{j=0}^\infty2^{j-n}e^{-c2^{j-n}}\|\Delta_jf^*\|_{L^p(\T)}\right)^p.
\label{20260816eq50}
\end{align}
To sum these estimates over $n$, define $\|\Delta_jf^*\|_{L^p(\T)}:=0$ if $j<0$. Then by \eqref{20260816eq50} and Young's inequality, 
\begin{align} \label{20260816eq43}
\mathfrak L_2
&=\sum_{n=1}^\infty \int_{A_n}M_p(f',r)^p(1-r)^{p-1}\,dr \nonumber \\
& \lesssim \sum_{n=1}^\infty \left(\sum_{j=0}^\infty2^{j-n}e^{-c2^{j-n}}\|\Delta_jf^*\|_{L^p(\T)}\right)^p \nonumber \\
& \le \sum_{n \in \Z} \left(\sum_{j \in \Z} 2^{j-n}e^{-c2^{j-n}}\|\Delta_jf^*\|_{L^p(\T)}\right)^p \nonumber \\
&= \left\| \left\{ \left\{2^{-\nu} e^{-c2^{-\nu}} \right\}_{\nu \in \Z} * \left\{ \|\Delta_jf^*\|_{L^p(\T)} \right\}_{j \in \Z} \right\} \right\|_{\ell^p}^p \nonumber \\
& \le \left(\sum_{\nu \in \Z} 2^{-\nu}e^{-c2^{-\nu}} \right)^p \cdot \left(\sum_{j=0}^\infty \|\Delta_jf^*\|_{L^p(\T)}^p \right)  \nonumber \\
& \lesssim \sum_{j=0}^\infty \|\Delta_jf^*\|_{L^p(\T)}^p. 
\end{align}
Here, in the last estimate above, we have used the fact that
$$
\sum_{\nu \in \Z} 2^{-\nu}e^{-c2^{-\nu}}  \le \sum_{\nu \ge 0} 2^{-\nu}+\sum_{\nu<0} 2^{-\nu}e^{-c2^{-\nu}} <+\infty. 
$$
The estimate of $\mathfrak L_2$ is complete. 

\medskip 

Finally, plugging \eqref{20260816eq43A} and \eqref{20260816eq43} back to \eqref{20260816eq43B}, we conclude that
$$
\int_\D|f'(z)|^p(1-|z|^2)^{p-1}\,dA(z) \lesssim \mathfrak L_1+\mathfrak L_2\lesssim\sum_{j=0}^\infty\|\Delta_jf^*\|_{L^p(\T)}^p.
$$
This proves the desired estimate \eqref{20260816eq46}, and hence completes the proof of  
\eqref{20260816eq45}.

\medskip
\noindent\textbf{Step II.} We next prove the reverse inequality
\begin{equation}\label{20260816eq54}
\sum_{j=0}^\infty\|\Delta_jf^*\|_{L^p(\T)}^p\lesssim |f(0)|^p+\int_\D|f'(z)|^p(1-|z|^2)^{p-1}\,dA(z).
\end{equation}
Write
\begin{equation}\label{20260817eq10}
\textrm{LHS of \eqref{20260816eq54}} =\|\Delta_0f^*\|_{L^p(\T)}^p+\sum_{n=1}^\infty\|\Delta_nf^*\|_{L^p(\T)}^p.
\end{equation}

We first estimate the term corresponding to $j=0$ in \eqref{20260817eq10}. Recall from \eqref{20260817eq05} that 
$$
\Delta_0f^*(t)=a_0+a_1e^{2\pi it},
$$
and hence by submean-value property, 
\begin{align}\label{20260817eq12}
\|\Delta_0f^*\|_{L^p(\T)}^p
& \lesssim |a_0|^p+|a_1|^p \nonumber \\
&=|f(0)|^p+|f'(0)|^p \nonumber \\
& \lesssim |f(0)|^p+\int_{|z|<1/2}|f'(z)|^p\,dA(z) \nonumber \\ 
& \lesssim |f(0)|^p+\int_\D|f'(z)|^p(1-|z|^2)^{p-1}\,dA(z).
\end{align}

We next estimate the terms corresponding to $n\geq1$ in \eqref{20260817eq10}. Fix $n\geq1$, $r\in A_n=\left[1-2^{-n},1-2^{-n-1}\right)$, and define $T_{\widetilde{m}_{n, r}}$ to be the periodic Fourier multiplier with symbol
$$
\widetilde m_{n,r}(\xi):=
\begin{cases}
\dfrac{2^n\varphi_n(\xi)}{\xi r^\xi},&\xi>0,\\[6pt]
0,&\xi\leq0.
\end{cases}
$$
Since $\varphi_n$ is supported away from the origin, $\widetilde m_{n,r}$ is smooth. We shall verify that $\widetilde m_{n,r}$ satisfies the hypotheses of the strong Marcinkiewicz multiplier theorem, with bounds uniform in $n\geq1$ and $r\in A_n$.

\vspace{0.1cm}

\noindent\textit{\underline{Claim III}:
\begin{equation}\label{20260817eq01}
|\widetilde m_{n,r}(\xi)|\lesssim1,\qquad \xi\in\R.
\end{equation}}

\vspace{0.1cm}

It suffices to consider $\xi\in\supp \ \widetilde m_{n,r}$, for which $\xi\simeq2^n$. Moreover, since $r\in A_n$, a direct computation gives 
\begin{equation} \label{20260817eq30}
r^\xi \simeq 1,
\end{equation} 
uniformly in $n, r$ and $\xi$. Therefore, 
$$
|\widetilde m_{n,r}(\xi)| \lesssim \frac{2^n|\varphi_n(\xi)|}{\xi r^\xi}\lesssim1,
$$
which proves \eqref{20260817eq01}.

\vspace{0.1cm}

\noindent\textit{\underline{Claim IV}:
\begin{equation}\label{20260817eq02}
\sum_{k\in\Z}|\widetilde m_{n,r}(k+1)-\widetilde m_{n,r}(k)|\lesssim1.
\end{equation}}

\vspace{0.1cm}

For $\xi>0$, a direct computation gives
\begin{align*}
\widetilde m_{n,r}'(\xi)
&=2^nr^{-\xi}\left(\frac{\varphi_n'(\xi)}{\xi}-\frac{\varphi_n(\xi)}{\xi^2}-\frac{\log r}{\xi}\varphi_n(\xi)\right).
\end{align*}
Since $|\varphi_n'(\xi)|\lesssim 2^{-n}$, and $|\log r|\lesssim1-r\lesssim2^{-n}$ for $r \in A_n$, the above calculation together with \eqref{20260817eq30} gives
\begin{align*}
|\widetilde m_{n,r}'(\xi)|
&\lesssim2^n\left(\frac{1}{2^{2n}}+\frac{1}{2^{2n}}+\frac{1}{2^{2n}}\right)
\lesssim2^{-n}.
\end{align*}
Since $\supp \ \widetilde m_{n,r}$ is contained in an interval of length $O(2^n)$, it follows that
\begin{align*}
\sum_{k\in\Z}|\widetilde m_{n,r}(k+1)-\widetilde m_{n,r}(k)|
&\leq\sum_{k\in\Z}\int_k^{k+1}|\widetilde m_{n,r}'(\xi)|\,d\xi\\
&=\int_\R|\widetilde m_{n,r}'(\xi)|\,d\xi \lesssim 1.
\end{align*}
This proves \eqref{20260817eq02}.

\vspace{0.1cm}

As a consequence, \eqref{20260817eq01} and \eqref{20260817eq02} verify the hypotheses of the strong Marcinkiewicz multiplier theorem \cite[Theorem~8.2.1]{EdwardsGaudry1977}. It follows that
\begin{equation}\label{20260816eq55}
\|T_{\widetilde m_{n,r}}h\|_{L^p(\T)}\lesssim_p\|h\|_{L^p(\T)},
\end{equation}
uniformly for $n\geq1$ and $r\in A_n$.

\vspace{0.1cm}

We apply \eqref{20260816eq55} to the function
$$
Rf(re^{2\pi it})=\sum_{k\geq1}kr^ka_ke^{2\pi ikt}.
$$
By the definition of $\widetilde m_{n,r}$,
\begin{align*}
T_{\widetilde m_{n,r}}\big(Rf(re^{2\pi i(\cdot)})\big)(e^{2\pi it})
&=\sum_{k\geq1}\widetilde m_{n,r}(k)kr^ka_ke^{2\pi ikt}  \\
&=2^n\sum_{k\geq1}\varphi_n(k)a_ke^{2\pi ikt}\\
&=2^n\Delta_nf^*(t).
\end{align*}
Consequently, by \eqref{20260816eq55}, 
\begin{align*}
2^n\|\Delta_nf^*\|_{L^p(\T)}
&=\left\| T_{\widetilde m_{n,r}}\big(Rf(re^{2\pi i(\cdot)})\big)(e^{2\pi i(\cdot)}) \right\|_{L^p(\T)} \nonumber   \\
&\lesssim\|Rf(re^{2\pi i(\cdot)})\|_{L^p(\T)}\nonumber\\
&=rM_p(f',r) \leq M_p(f',r),
\end{align*}
and hence
\begin{equation}\label{20260817eq15}
\|\Delta_nf^*\|_{L^p(\T)}\lesssim2^{-n}M_p(f',r), \qquad \forall r \in A_n.
\end{equation}
Therefore, 
\begin{align}
\|\Delta_nf^*\|_{L^p(\T)}^p
&=\frac{1}{|A_n|}\int_{A_n}\|\Delta_nf^*\|_{L^p(\T)}^p\,dr\nonumber\\
&\lesssim\frac{2^{-np}}{|A_n|}\int_{A_n}M_p(f',r)^p\,dr\nonumber\\
&\simeq2^{-n(p-1)}\int_{A_n}M_p(f',r)^p\,dr\nonumber\\
&\simeq\int_{A_n}M_p(f',r)^p(1-r)^{p-1}\,dr,
\label{20260817eq16}
\end{align}
where we have used $|A_n|\simeq2^{-n}$ and $1-r\simeq2^{-n}$ on $A_n$. Summing \eqref{20260817eq16} over $n\geq1$, we obtain
\begin{align}
\sum_{n=1}^\infty\|\Delta_nf^*\|_{L^p(\T)}^p
&\lesssim\sum_{n=1}^\infty\int_{A_n}M_p(f',r)^p(1-r)^{p-1}\,dr\nonumber\\
&=\int_{1/2}^1M_p(f',r)^p(1-r)^{p-1}\,dr \nonumber \\
&\lesssim\int_{1/2}^1M_p(f',r)^p(1-r^2)^{p-1}r\,dr\nonumber\\
&\lesssim\int_\D|f'(z)|^p(1-|z|^2)^{p-1}\,dA(z).  \label{20260817eq19}
\end{align}
Finally, substituting \eqref{20260817eq12} and \eqref{20260817eq19} into \eqref{20260817eq10} gives
$$
\sum_{j=0}^\infty\|\Delta_jf^*\|_{L^p(\T)}^p\lesssim |f(0)|^p+\int_\D|f'(z)|^p(1-|z|^2)^{p-1}\,dA(z).
$$
This proves \eqref{20260816eq54}. Together with {\bf Step I}, the proof of \eqref{20260818eq01} is complete. 

\medskip
\noindent\textbf{Step III.} It remains to establish \eqref{20260818eq02}. We first show that analytic trigonometric polynomials are dense in $B_{p,p,+}^0(\T)$. Let $h\in B_{p,p,+}^0(\T)$ and set
$$
h_N:=\sum_{j=0}^N\Delta_jh.
$$
Since each $\Delta_jh$ has finite Fourier support and $h$ has no negative Fourier coefficients, $h_N$ is an analytic trigonometric polynomial. By the finite overlap of the supports of the functions $\{\varphi_j\}_{j\geq0}$ and the uniform $L^p(\T)$-boundedness of the corresponding smooth Fourier multipliers,
$$
\|h-h_N\|_{B_{p,p}^0(\T)}^p
\lesssim\sum_{j\geq N-C}\|\Delta_jh\|_{L^p(\T)}^p\longrightarrow0
$$
as $N\to\infty$, where $C>0$ is an absolute constant. Thus, analytic trigonometric polynomials are dense in $B_{p,p,+}^0(\T)$.

It now follows from \eqref{20260818eq01} and the density of analytic polynomials in $\calD_{p-1}^p$ that the map $f\mapsto f^*$ extends uniquely to a bounded linear map
$$
\operatorname{Tr}:\calD_{p-1}^p\longrightarrow B_{p,p,+}^0(\T).
$$
The reverse inequality in \eqref{20260818eq01} shows that $\operatorname{Tr}$ is bounded below, and hence is injective and has closed range. Moreover, $\textrm{ran}(\operatorname{Tr})$ contains every analytic trigonometric polynomial. Since such polynomials are dense in $B_{p,p,+}^0(\T)$, we conclude that $\textrm{ran}(\operatorname{Tr})=B_{p,p,+}^0(\T)$. Therefore, $\operatorname{Tr}$ is an isomorphism whose inverse is also bounded.

Finally, let
$$
f(z)=\sum_{n=0}^\infty a_nz^n\in\calD_{p-1}^p,
$$
and choose analytic polynomials
$$
f_k(z)=\sum_{n=0}^{N_k}a_n^{(k)}z^n
$$
such that $f_k\to f$ in $\calD_{p-1}^p$. By the continuity of $\operatorname{Tr}$,
$$
\operatorname{Tr}f_k\longrightarrow\operatorname{Tr}f
\qquad\textrm{in}\qquad B_{p,p}^0(\T).
$$
For each fixed $n\in\Z$, observe that the functional $h\mapsto\widehat h(n)$ is continuous on $B_{p,p}^0(\T)$. Indeed, one may choose $j\geq0$ such that $\varphi_j(n)\neq0$, and then
$$
|\widehat h(n)|
=\frac{|\widehat{\Delta_jh}(n)|}{|\varphi_j(n)|}
\leq\frac{\|\Delta_jh\|_{L^1(\T)}}{|\varphi_j(n)|}
\lesssim_n\|h\|_{B_{p,p}^0(\T)}.
$$
It follows that
\begin{equation}\label{20260818eq20}
\widehat{\operatorname{Tr}f_k}(n)\longrightarrow\widehat{\operatorname{Tr}f}(n)
\qquad \forall n\in\Z.
\end{equation}

On the other hand, since norm convergence in $\calD_{p-1}^p$ implies locally uniform convergence on $\D$, by Cauchy's formula, we have
$$
a_n^{(k)}\longrightarrow a_n
\qquad \forall n\geq0.
$$
Since
$$
\widehat{\operatorname{Tr}f_k}(n)=
\begin{cases}
a_n^{(k)},&0\leq n\leq N_k,\\
\\
0,&n<0,
\end{cases}
$$
passing to the limit in \eqref{20260818eq20} gives
$$
\widehat{\operatorname{Tr}f}(n)=
\begin{cases}
a_n,&n\geq0,\\
\\
0,&n<0.
\end{cases}
$$
These Fourier coefficients determine $\operatorname{Tr}f$ uniquely as a distribution on $\T$. This completes the proof.
\end{proof}

Having shown the equivalence between $\calD_{p-1}^p$ and $B_{p, p,+}^0(\T)$, we now turn to the main result of this section. For $h\in\mathcal D'(\T)$, define its Poisson extension by
\begin{equation}\label{20260819eq00}
P[h](re^{2\pi it})
:=
\sum_{n\in\Z}r^{|n|}\widehat h(n)e^{2\pi int},
\qquad 0\leq r<1.
\end{equation}

\begin{prop}\label{20260819prop01}
Let $1<p<\infty$, and let $\mu$ be a positive Borel measure on $\D$. Then \eqref{Dirichletembedding01} holds if and only if
\begin{equation}\label{20260819eq01}
\int_\D|P[h](z)|^p\,d\mu(z)
\leq C\|h\|_{B_{p,p}^0(\T)}^p, \qquad h \in B_{p, p}^0(\T).
\end{equation}
Moreover, if $\widetilde{C}_{best}$ denotes the optimal constant in \eqref{20260819eq01}, then
$$
\widetilde{C}_{best}\simeq_p C_{best},
$$
where $C_{best}$ is the optimal constant in \eqref{Dirichletembedding01}, as defined in Theorem~\ref{20260813thmmain}.
\end{prop}

\begin{proof}
Suppose first that \eqref{20260819eq01} holds, and we have to show \eqref{Dirichletembedding01}, that is
$$
\int_{\D} |f(z)|^p d\mu(z) \lesssim \left\|f\right\|_{\calD_{p-1}^p}^p, \qquad f \in \calD_{p-1}^p.
$$
Let
$$
f(z)=\sum_{n=0}^\infty a_nz^n\in\calD_{p-1}^p
$$
and hence by Proposition~\ref{20260813eq01}, $f^*=\operatorname{Tr} f \in B_{p,p,+}^0(\T)$ with
$$
\widehat{f^*}(n)=
\begin{cases}
a_n,&n\geq0,\\
\\
0,&n<0.
\end{cases}
$$
Therefore, 
\begin{equation} \label{20260819eq40}
P[f^*]=f. 
\end{equation} 
Indeed, for $z=re^{2\pi it}\in\D$, one has
$$
P[f^*](re^{2\pi it})=\sum_{n=0}^\infty r^na_ne^{2\pi int}=\sum_{n=0}^\infty a_n\left(re^{2\pi it}\right)^n=f(re^{2\pi it}).
$$
Now, using \eqref{20260819eq01} and Proposition~\ref{20260813eq01}, we obtain
\begin{align*}
\int_\D|f(z)|^p\,d\mu(z)
&=\int_\D|P[f^*](z)|^p\,d\mu(z)\\
&\leq\widetilde{C}_{best}\|f^*\|_{B_{p,p}^0(\T)}^p\\
&\lesssim_p\widetilde{C}_{best}\|f\|_{\calD_{p-1}^p}^p.
\end{align*}
Thus, \eqref{Dirichletembedding01} holds and
\begin{equation}\label{20260819eq02}
C_{best}\lesssim_p\widetilde{C}_{best}.
\end{equation}

Conversely, suppose that \eqref{Dirichletembedding01} holds, and let $h\in B_{p,p}^0(\T)$. Decompose
$$
h=h_++h_-,
$$
where
$$
\widehat{h_+}(n):=
\begin{cases}
\widehat h(n),&n\geq0,\\
\\
0,&n<0,
\end{cases}
\qquad \textrm{and} \qquad 
\widehat{h_-}(n):=
\begin{cases}
0,&n\geq0,\\
\\
\widehat h(n),&n<0.
\end{cases}
$$
Then
$$
h_+\in B_{p,p,+}^0(\T)
\qquad\textrm{and}\qquad
\overline{h_-}\in B_{p,p,+}^0(\T).
$$
Moreover, observe that another application of the strong Marcinkiewicz multiplier theorem \cite[Theorem~8.2.1]{EdwardsGaudry1977} shows that the projections
$$
\mathcal P_+h:=h_+
\qquad\textrm{and}\qquad
\mathcal P_-h:=h_-
$$
are bounded on $L^p(\T)$. Therefore, 
\begin{align}
\|h_+\|_{B_{p,p}^0(\T)}^p+\|h_-\|_{B_{p,p}^0(\T)}^p
&=\sum_{j=0}^\infty\left(
\|\Delta_jh_+\|_{L^p(\T)}^p+\|\Delta_jh_-\|_{L^p(\T)}^p
\right)\nonumber\\
&=\sum_{j=0}^\infty\left(
\|\Delta_j \mathcal P_+h\|_{L^p(\T)}^p+\|\Delta_j\mathcal P_-h\|_{L^p(\T)}^p
\right)\nonumber\\
&=\sum_{j=0}^\infty\left(
\| \mathcal P_+ \Delta_j h\|_{L^p(\T)}^p+\|\mathcal P_- \Delta_jh\|_{L^p(\T)}^p
\right)\nonumber\\
&\lesssim_p\sum_{j=0}^\infty\|\Delta_jh\|_{L^p(\T)}^p=\|h\|_{B_{p,p}^0(\T)}^p.
\label{20260819eq03}
\end{align}

By Proposition~\ref{20260813eq01}, there exist unique functions $f,g\in\calD_{p-1}^p$ such that
$$
\operatorname{Tr}f=h_+
\qquad\textrm{and}\qquad
\operatorname{Tr}g=\overline{h_-}.
$$
By the same argument as in \eqref{20260819eq40}, we obtain
$$
P[h_+]=f
\qquad\textrm{and}\qquad
P[h_-]=\overline g.
$$
Consequently,
\begin{equation}\label{20260819eq04}
P[h]=f+\overline g.
\end{equation}
Using \eqref{20260819eq04}, \eqref{Dirichletembedding01}, Proposition~\ref{20260813eq01}, and \eqref{20260819eq03}, we obtain
\begin{align*}
\int_\D|P[h](z)|^p\,d\mu(z)
&\lesssim_p\int_\D|f(z)|^p\,d\mu(z)
+\int_\D|g(z)|^p\,d\mu(z)\\
&\leq C_{best}\left(
\|f\|_{\calD_{p-1}^p}^p+\|g\|_{\calD_{p-1}^p}^p
\right)\\
&\lesssim_p C_{best}\left(
\|h_+\|_{B_{p,p}^0(\T)}^p
+\|\overline{h_-}\|_{B_{p,p}^0(\T)}^p
\right)\\
&\lesssim_p C_{best}\|h\|_{B_{p,p}^0(\T)}^p.
\end{align*}
Here, we have used the fact that $\|\overline{h_-}\|_{B_{p,p}^0(\T)}=\|h_-\|_{B_{p,p}^0(\T)}$. Indeed, since the functions $\varphi_j$ are even, a direct computation yields $\Delta_j(\overline{h_-})=\overline{\Delta_j h_{-}}$. Hence, $\left\|\Delta_j(\overline{h_-}) \right\|_{L^p(\T)}=\left\|\Delta_j h_- \right\|_{L^p(\T)}$. The desired claim then follows by taking the $p$-th power on both sides of the preceding equation and summing over $j$.

\vspace{0.1cm}

Therefore, \eqref{20260819eq01} holds and
\begin{equation}\label{20260819eq05}
\widetilde{C}_{best}\lesssim_p C_{best}.
\end{equation}
Combining \eqref{20260819eq02} and \eqref{20260819eq05} completes the proof.
\end{proof}

In view of Proposition~\ref{20260819prop01}, the original Dirichlet embedding problem \eqref{Dirichletembedding01} is equivalent to the boundedness of the Poisson extension 
\begin{equation} \label{Harmonicembedding}
P:B_{p,p}^0(\T)\longrightarrow L^p(\mu).
\end{equation}
This finishes the first reduction.

\bigskip 

\section{Reduction II: from harmonic embedding to Whitney embedding}\label{Sec03}

The goal of this section is to make a second reduction from the harmonic embedding problem \eqref{Harmonicembedding} to a Whitney embedding problem, where the methods of dyadic harmonic analysis enter more naturally.

\subsection{The Whitney model and its tree formulation}\label{20260820subsec01}

For $n\geq0$, let
$$
\calD_n:=\{I\in\calD:|I|=2^{-n}\}
$$
be the $n$-th generation of $\calD$, and let $\mathbb E_n$ denote the conditional expectation with respect to the partition $\calD_n$. Thus, for $h\in L^1(\T)$,
$$
\mathbb E_nh
=
\sum_{I\in\calD_n}
\left(\frac{1}{|I|}\int_Ih(t)\,dt\right)\one_I.
$$
For each $I\in\calD$, define the $L^\infty$-normalized Haar function
$$
h_I:=\one_{I_+}-\one_{I_-},
$$
where we recall that $I_{+}$ and $I_{-}$ are the two dyadic children of $I$, ordered counterclockwise. We say that $h$ is a \emph{Haar polynomial} on $\T$ if it is of the form
\begin{equation}\label{20260820eq01}
h=\widehat h(0)+\sum_{I\in\calD}a_Ih_I, \qquad a_I:=\frac{1}{|I|}\int_Ih(t)h_I(t)\,dt,
\end{equation}
where only finitely many coefficients $a_I$ are nonzero. We also write
$$
{\bf d}_nh:=\mathbb E_{n+1}h-\mathbb E_nh
=\sum_{I\in\calD_n}a_Ih_I.
$$
\begin{lem}\label{20260820lemHaar}
Let $1<p<\infty$. For every Haar polynomial $h$ given by \eqref{20260820eq01}, one has
\begin{equation}\label{20260820eq02}
\|h\|_{B_{p,p}^0(\T)}^p
\simeq_p|\widehat h(0)|^p+\sum_{I\in\calD}|a_I|^p|I|.
\end{equation}
Equivalently,
\begin{equation}\label{20260820eq02A}
\|h\|_{B_{p,p}^0(\T)}^p
\simeq_p|\widehat h(0)|^p+\sum_{n=0}^\infty\|{\bf d}_nh\|_{L^p(\T)}^p.
\end{equation}
Moreover, Haar polynomials are dense in $B_{p,p}^0(\T)$.
\end{lem}

\begin{proof}
Applying the standard Haar characterization for Besov functions
\cite[(1.7)]{GarrigosSeegerUllrich2023} with $s=0$ and $q=p$ in this periodic setting, we see that for any $h \in B_{p, p}^0(\T)$, 
\begin{equation}\label{20260820eq02B}
\|h\|_{B_{p,p}^0(\T)}^p
\simeq_p|\widehat h(0)|^p
+\sum_{n=0}^\infty2^{-n} \left(\sum_{I\in\calD_n}
\left|2^n\langle h,h_I\rangle\right|^p \right).
\end{equation}
Since $I\in\calD_n$, we have $|I|=2^{-n}$ and
$$
2^n\langle h,h_I\rangle
=\frac{1}{|I|}\int_Ih(t)h_I(t)\,dt
=a_I.
$$
This together with \eqref{20260820eq02B} gives
\begin{align*}
\|h\|_{B_{p,p}^0(\T)}^p
&\simeq_p|\widehat h(0)|^p
+\sum_{n=0}^\infty2^{-n} \left(\sum_{I\in\calD_n}|a_I|^p \right)\\
&=|\widehat h(0)|^p
+\sum_{n=0}^\infty\sum_{I\in\calD_n}|a_I|^p|I|\\
&=|\widehat h(0)|^p+\sum_{I\in\calD}|a_I|^p|I|,
\end{align*}
which proves \eqref{20260820eq02}.

\vspace{0.1cm}

Moreover, since
$$
{\bf d}_nh=\sum_{I\in\calD_n}a_Ih_I,
$$
the pairwise disjointness of the intervals in $\calD_n$ gives
$$
\|{\bf d}_nh\|_{L^p(\T)}^p
=\sum_{I\in\calD_n}|a_I|^p|I|,
$$
and \eqref{20260820eq02A} follows by summing the identity over $n\geq0$. Finally, the density of Haar polynomials in $B_{p,p}^0(\T)$ follows from a standard truncation argument, whose details are omitted.
\end{proof}

\vspace{0.1cm}

For $I\in\calD$, let
$$
Q_I^{\textrm{up}}:=Q_I\setminus\left(Q_{I_+}\cup Q_{I_-}\right)
$$
be the \emph{upper Carleson tent} associated with $I$. Therefore, 
\begin{equation}\label{20260820eq03}
Q_I=\bigsqcup_{\substack{J\in\calD \\ J\subseteq I}}Q_J^{\textrm{up}},
\end{equation}
and 
$$
\mu(Q_I)=\sum_{\substack{J\in\calD \\ J\subseteq I}}\mu(Q_J^{\textrm{up}}).
$$

For a Haar polynomial $h$, we define its \emph{dyadic Whitney extension} by
\begin{equation}\label{20260820eq06}
\mathcal W_{\calD}h(z):=\frac{1}{|I|}\int_Ih(t)\,dt,
\qquad z\in Q_I^{\textrm{up}},\quad I\in\calD.
\end{equation}
Equivalently, if $I\in\calD_n$ and $z=re^{2\pi it}\in Q_I^{\textrm{up}}$, then
$$
\mathcal W_{\calD}h(z)=\mathbb E_nh(t).
$$
Since $\mathbb E_nh$ is constant on every $I\in\calD_n$, the right-hand side of \eqref{20260820eq06} is constant on $Q_I^{\textrm{up}}$. The following lemma is straightforward. 

\begin{lem}\label{20260820lem02}
Let $h$ be a Haar polynomial given by \eqref{20260820eq01}. Then, for every $J\in\calD$, the function $\mathcal W_{\calD}h$ is constant on $Q_J^{\textrm{up}}$, and
\begin{equation}\label{20260820eq07}
\left.\mathcal W_{\calD}h\right|_{Q_J^{\textrm{up}}}
=\widehat h(0)+\sum_{\substack{I\in\calD \\ I\supsetneq J}} a_I \left(\left.h_I\right|_J\right),
\end{equation}
where
\begin{equation} \label{20260821eq02}
\left.h_I\right|_J=
\begin{cases}
1,&J\subseteq I_+,\\
\\
-1,&J\subseteq I_-.
\end{cases}
\end{equation} 
\end{lem}

\begin{proof}
By the definition of $\mathcal W_{\calD}h$, we have
\begin{align} \label{20260820eq23A}
\left.\mathcal W_{\calD}h\right|_{Q_J^{\textrm{up}}}=\frac{1}{|J|}\int_Jh(t)\,dt=\widehat h(0)+\sum_{I\in\calD}a_I\frac{1}{|J|}\int_Jh_I(t)\,dt.
\end{align}
Since $I$ and $J$ are dyadic intervals, they are either disjoint or one is contained in the other. If $I\cap J=\varnothing$, then
$$
\int_Jh_I(t)\,dt=0.
$$
If $I\subseteq J$, then
$$
\int_Jh_I(t)\,dt=\int_Ih_I(t)\,dt=|I_+|-|I_-|=0.
$$
Finally, if $J\subsetneq I$, then $J$ is contained in exactly one of the two dyadic children of $I$. Hence $h_I$ is constant on $J$ and
$$
\frac{1}{|J|}\int_Jh_I(t)\,dt
=\left.h_I\right|_J
=
\begin{cases}
1,&J\subseteq I_+,\\
\\
-1,&J\subseteq I_-.
\end{cases}
$$
Therefore, only the third case above contributes to the sum \eqref{20260820eq23A}, which proves \eqref{20260820eq07}.
\end{proof}

We shall consider the \emph{Whitney embedding}
\begin{equation}\label{20260820eq08}
\int_\D|\mathcal W_{\calD}h(z)|^p\,d\mu(z)
\leq \widehat{C}_{best} \|h\|_{B_{p,p}^0(\T)}^p
\end{equation}
for every Haar polynomial $h$, where $\widehat{C}_{best}$ denotes the optimal constant in \eqref{20260820eq08}. Since $\mathcal W_{\calD}h$ is constant on each upper Carleson tent, \eqref{20260820eq02} and \eqref{20260820eq07} show that \eqref{20260820eq08} is equivalent to the following tree estimate
\begin{equation}\label{20260820eq09}
\sum_{J\in\calD}
\left|\widehat h(0)+\sum_{\substack{I\in\calD \\ I\supsetneq J}}a_I \left(\left.h_I\right|_J\right)\right|^p
\mu(Q_J^{\textrm{up}})
\lesssim_p|\widehat h(0)|^p+\sum_{I\in\calD}|a_I|^p|I|.
\end{equation}
The rest of this section is devoted to proving that the harmonic embedding \eqref{20260819eq01} and the Whitney embedding \eqref{20260820eq08} are equivalent.

\subsection{The equivalence between the harmonic and Whitney embeddings}\label{20260820subsec02}

For any dyadic system $\calD$ on $\T$, denote
\begin{equation}\label{20260820eq10}
\left\|\mu\right\|_{\mathcal{CM}, \calD}:=\sup_{I\in\calD}\frac{\mu(Q_I)}{|I|}.
\end{equation}
We have the following. 

\begin{lem}\label{20260820lem01}
Let $1<p<\infty$. Then the following statements hold.
\begin{enumerate}
    \item If the harmonic embedding \eqref{20260819eq01} holds, then
\begin{equation}\label{20260820eq11}
\left\|\mu\right\|_{\mathcal{CM}, \calD}\lesssim_p\widetilde C_{best}.
\end{equation}
\item 
If the Whitney embedding \eqref{20260820eq08} holds, then
\begin{equation}\label{20260820eq12}
\left\|\mu\right\|_{\mathcal{CM}, \calD}\lesssim_p \widehat{C}_{best}.
\end{equation}
\end{enumerate}
\end{lem}

\begin{proof}
We begin with recording an estimate that will be used in both parts. Fix $I\in\calD$, and suppose that $I\in\calD_m$. By direct calculation,
$\widehat{\one_I}(0)=|I|$ and, for every $J\in\calD$,
$$
\frac{1}{|J|}\int_J\one_I(t)h_J(t)\,dt
=
\begin{cases}
\dfrac{|I|}{|J|},&I\subseteq J_+;\\
-\dfrac{|I|}{|J|},&I\subseteq J_-;\\
0,&J\not\supsetneq I.
\end{cases}
$$
It follows from \eqref{20260820eq02} that
\begin{align}
\|\one_I\|_{B_{p,p}^0(\T)}^p
&\simeq_p|I|^p+\sum_{\substack{J\in\calD\\J\supsetneq I}}
\left(\frac{|I|}{|J|}\right)^p|J|\nonumber\\
&=|I|^p+|I|^p\sum_{k=1}^m(2^k|I|)^{1-p}\nonumber\\
&=|I|^p+|I|\sum_{k=1}^m2^{-k(p-1)}
\lesssim_p|I|.
\label{20260820eq13}
\end{align}

We first prove the assertion $(1)$. Suppose that the harmonic embedding \eqref{20260819eq01} holds. Recall that for $0 \le r<1$, the Poisson kernel satisfies
\begin{align}\label{20260820eq14X}
P_r(u)
&:=\frac{1-r^2}{1-2r\cos(2\pi u)+r^2} =\frac{1-r^2}{(1-r)^2+4r\sin^2(\pi u)} \nonumber \\ 
&\simeq\frac{1-r}{(1-r)^2+d_{\T}(u,0)^2},
\end{align}
where $d_{\T}$ denotes the periodic distance on $\T$.

Let $I \subseteq \T$ and $z=re^{2\pi it}\in Q_I$. Then $t\in I$ and $1-r\leq|I|$. Note that the point $t$ divides $I$ into two subarcs whose lengths sum to $|I|$, hence one of these subarcs has length at least $|I|/2$. Therefore, by \eqref{20260820eq14X},
\begin{align}  \label{20260820eq14}
P[\one_I](re^{2\pi it})
&=\int_I P_r(t-s)ds  \simeq \int_I \frac{1-r}{(1-r)^2+d_{\T}(t-s,0)^2} ds \nonumber \\ 
&\gtrsim\int_0^{\frac{|I|}{2}}\frac{1-r}{(1-r)^2+s^2}\,ds =\int_0^{\frac{|I|}{2(1-r)}}\frac{1}{1+u^2}\,du
\gtrsim1,
\end{align}
where in the last estimate above, we have used the fact that $|I|/(1-r) \ge 1$. Combining now \eqref{20260820eq14}, \eqref{20260819eq01} and \eqref{20260820eq13}, we see that
\begin{align*}
\mu(Q_I)
&\lesssim\int_{Q_I}|P[\one_I](z)|^p\,d\mu(z)\\
&\leq\widetilde C_{\mathrm{best}}\|\one_I\|_{B_{p,p}^0(\T)}^p
\lesssim_p\widetilde C_{\mathrm{best}}|I|.
\end{align*}
Taking the supremum over $I\in\calD$ proves \eqref{20260820eq11}.

\medskip

Next, we prove assertion $(2)$. Suppose that the Whitney embedding \eqref{20260820eq08} holds. By \eqref{20260820eq03} and the definition of $\mathcal W_{\calD}$, we have
\begin{equation} \label{20260821eq01}
\mathcal W_{\calD}\one_I(z)=1,\qquad z\in Q_I.
\end{equation} 
Moreover, since 
$$
\one_I(t)
=|I|+\sum_{J\in\calD, \; J\supsetneq I}
\frac{|I|}{|J|} \left(\left.h_J\right|_I \right)h_J(t), \qquad t \in \T, 
$$
where we recall that $\left.h_J\right|_I$ is defined in \eqref{20260821eq02}, 
one has that $\one_I$ is a Haar polynomial, and hence \eqref{20260821eq01}, \eqref{20260820eq08} and \eqref{20260820eq13} imply that
\begin{align*}
\mu(Q_I)
&=\int_{Q_I}|\mathcal W_{\calD}\one_I(z)|^p\,d\mu(z)\\
&\leq\widehat C_{\mathrm{best}}\|\one_I\|_{B_{p,p}^0(\T)}^p
\lesssim_p\widehat C_{\mathrm{best}}|I|.
\end{align*}
Taking the supremum over $I\in\calD$ proves \eqref{20260820eq12}.
\end{proof}

We next compare the Poisson extension with the Whitney extension.

\begin{thm}\label{20260820thm01}
Let $1<p<\infty$, and let $\mu$ be a positive Borel measure on $\D$ such that $\left\|\mu\right\|_{\mathcal{CM}, \calD}<\infty$. Then
\begin{equation}\label{20260820eq14Y}
\int_\D|P[h](z)-\mathcal W_{\calD}h(z)|^p\,d\mu(z)
\lesssim_p \left\|\mu\right\|_{\mathcal{CM}, \calD}\|h\|_{B_{p,p}^0(\T)}^p
\end{equation}
for every Haar polynomial $h$. Consequently, the operator
$$
h\longmapsto P[h]-\mathcal W_{\calD}h,
$$
initially defined on Haar polynomials, extends uniquely to a bounded linear operator from $B_{p,p}^0(\T)$ into $L^p(\mu)$.
\end{thm}

\begin{proof}
Let $h$ be a Haar polynomial and write
$$
h=\widehat h(0)+\sum_{k=0}^\infty{\bf d}_kh.
$$
Note that the above sum is finite. 

\vspace{0.1cm}

For $n,k\geq0$, define
$$
A_{n,k}(h):=
\left[\sum_{I\in\calD_n}|I|
\left(\sup_{z\in Q_I^{\textrm{up}}}
\left|P[{\bf d}_kh](z)-\left.\mathbb E_n({\bf d}_kh)\right|_I\right|^p \right) \right]^{1/p}.
$$
We claim that
\begin{equation}\label{20260820eq16}
A_{n,k}(h)\lesssim_p
\begin{cases}
2^{-(n-k)/p}\|{\bf d}_kh\|_{L^p(\T)},&k<n; \\
\\
2^{-(k-n)}\|{\bf d}_kh\|_{L^p(\T)},&k\geq n.
\end{cases}
\end{equation}

\medskip 

\noindent{\textit{\underline{Treatment of the case $k<n$}}}. We first make the following claim. 
\begin{equation}\label{20260820eq17}
\|\tau_u({\bf d}_kh)-{\bf d}_kh\|_{L^p(\T)}
\lesssim_p\min\left\{1,\bigl(2^kd_{\T}(u,0)\bigr)^{1/p}\right\}
\|{\bf d}_kh\|_{L^p(\T)},
\end{equation}
where $\tau_uf(t):=f(t-u)$. Suppose first that
$d_{\T}(u,0)\leq2^{-k-3}$. Since ${\bf d}_kh$ is constant on every interval in $\calD_{k+1}$, each of which has length $2^{-k-1}$, the function
$\tau_u({\bf d}_kh)-{\bf d}_kh$ can be nonzero only on subintervals of length at most $d_{\T}(u,0)$ adjacent to their endpoints. It follows that
$$
\|\tau_u({\bf d}_kh)-{\bf d}_kh\|_{L^p(\T)}^p
\lesssim_p2^kd_{\T}(u,0)\|{\bf d}_kh\|_{L^p(\T)}^p.
$$
If $d_{\T}(u,0)>2^{-k-3}$, then trivially there holds
$$
\|\tau_u({\bf d}_kh)-{\bf d}_kh\|_{L^p(\T)}
\leq2\|{\bf d}_kh\|_{L^p(\T)}.
$$
This proves \eqref{20260820eq17}.

\medskip
\noindent  We now return to the estimate of $A_{n,k}(h)$. Since $k+1\leq n$, we have
\begin{equation} \label{20260822eq01}
\mathbb E_n({\bf d}_kh)={\bf d}_kh.
\end{equation}
Fix $I\in\calD_n$. Then for any $x\in I$, and $z=re^{2\pi iy}\in Q_I^{\textrm{up}}$, one has
$$
1-r\simeq2^{-n}
\qquad\textrm{and}\qquad
d_{\T}(x,y)\leq2^{-n},
$$
which, together with  \eqref{20260820eq14X}, gives that
$$
P_r(y-t)\lesssim
\frac{2^{-n}}{2^{-2n}+d_{\T}(y,t)^2}
\lesssim
\frac{2^{-n}}{2^{-2n}+d_{\T}(x,t)^2}.
$$
Since the Poisson kernel has integral one, we obtain
\begin{align*}
|P[{\bf d}_kh](z)-{\bf d}_kh(x)|
&=\left|\int_\T P_r(y-t)
\bigl({\bf d}_kh(t)-{\bf d}_kh(x)\bigr)\,dt\right|\\
&\lesssim\int_\T
\frac{2^{-n}|{\bf d}_kh(t)-{\bf d}_kh(x)|}
{2^{-2n}+d_{\T}(x,t)^2}\,dt\\
&=\int_\T
\frac{2^{-n}|{\bf d}_kh(x-u)-{\bf d}_kh(x)|}
{2^{-2n}+d_{\T}(u,0)^2}\,du.
\end{align*}
Therefore,
\begin{align}
\sup_{z\in Q_I^{\textrm{up}}}
|P[{\bf d}_kh](z)-{\bf d}_kh(x)|
&\lesssim\int_\T
\frac{2^{-n}|{\bf d}_kh(x-u)-{\bf d}_kh(x)|}
{2^{-2n}+d_{\T}(u,0)^2}\,du.
\label{20260820eq18}
\end{align}
Since ${\bf d}_kh$ is constant on every $I\in\calD_n$, the definition of $A_{n,k}(h)$, \eqref{20260822eq01}, and \eqref{20260820eq18} give
\begin{align*}
A_{n,k}(h)^p
&=\sum_{I\in\calD_n}|I|
\left(\sup_{z\in Q_I^{\textrm{up}}}
\left|P[{\bf d}_kh](z)-\left.\mathbb E_n({\bf d}_kh)\right|_I\right|^p \right) \\
&=\sum_{I\in\calD_n}\int_I
\sup_{z\in Q_I^{\textrm{up}}}
|P[{\bf d}_kh](z)-{\bf d}_kh(x)|^p\,dx\\
&\lesssim\left\|
\int_\T
\frac{2^{-n}|{\bf d}_kh(\cdot-u)-{\bf d}_kh(\cdot)|}
{2^{-2n}+d_{\T}(u,0)^2}\,du
\right\|_{L^p(\T)}^p.
\end{align*}
By generalized Minkowski's inequality and \eqref{20260820eq17},
\begin{align}
A_{n,k}(h)
&\lesssim\int_\T\frac{2^{-n}}
{2^{-2n}+d_{\T}(u,0)^2}
\|\tau_u({\bf d}_kh)-{\bf d}_kh\|_{L^p(\T)}\,du\nonumber\\
&\lesssim_p\|{\bf d}_kh\|_{L^p(\T)}
\int_\T\frac{2^{-n}}
{2^{-2n}+d_{\T}(u,0)^2}
\min\left\{1,\bigl(2^kd_{\T}(u,0)\bigr)^{1/p}\right\}\,du.
\label{20260820eq19}
\end{align}
Write now the last integral above into two parts:
\begin{equation} \label{20260822eq02}
\int_\T\frac{2^{-n}}
{2^{-2n}+d_{\T}(u,0)^2}
\min\left\{1,\bigl(2^kd_{\T}(u,0)\bigr)^{1/p}\right\}\,du=\mathfrak L_3+\mathfrak L_4, 
\end{equation} 
where
$$
\mathfrak L_3:=\int_{d_{\T}(u,0)\leq2^{-k}} \frac{2^{-n}}
{2^{-2n}+d_{\T}(u,0)^2}
\min\left\{1,\bigl(2^kd_{\T}(u,0)\bigr)^{1/p}\right\}\,du.
$$
and 
$$
\mathfrak L_4:=\int_{d_{\T}(u,0)>2^{-k}} \frac{2^{-n}}
{2^{-2n}+d_{\T}(u,0)^2}
\min\left\{1,\bigl(2^kd_{\T}(u,0)\bigr)^{1/p}\right\}\,du.
$$

\medskip
\noindent\textbf{Estimate of $\mathfrak L_3$.} If
$d_{\T}(u,0)\leq2^{-k}$, then $\min\left\{1,\bigl(2^kd_{\T}(u,0)\bigr)^{1/p} \right\}=\bigl(2^kd_{\T}(u,0)\bigr)^{1/p}$. Therefore, applying 
the change of variables
$s=2^nd_{\T}(u,0)$, we have 
\begin{align*}
\mathfrak L_3
&=\int_{\{d_{\T}(u,0)\leq2^{-k}\}}
\frac{2^{-n}\bigl(2^kd_{\T}(u,0)\bigr)^{1/p}}
{2^{-2n}+d_{\T}(u,0)^2}\,du \\
&\lesssim_p2^{-(n-k)/p}
\int_0^\infty\frac{s^{1/p}}{1+s^2}\,ds\\
&\lesssim_p2^{-(n-k)/p}.
\end{align*}

\medskip
\noindent\textbf{Estimate of $\mathfrak L_4$.} In this case, $\min\left\{1,\bigl(2^kd_{\T}(u,0)\bigr)^{1/p} \right\}=1$. Therefore, 
\begin{align*}
\mathfrak L_4
&=\int_{\{d_{\T}(u,0)>2^{-k}\}}
\frac{2^{-n}}
{2^{-2n}+d_{\T}(u,0)^2}\,du \\
&\lesssim2^{-n}\int_{2^{-k}}^{1/2}\frac{ds}{s^2}\\
&\lesssim2^{-(n-k)}
\leq2^{-(n-k)/p}.
\end{align*}
Combining the estimates for both $\mathfrak L_3$ and $\mathfrak L_4$ with \eqref{20260820eq19}, we conclude that
$$
A_{n,k}(h)\lesssim_p
2^{-(n-k)/p}\|{\bf d}_kh\|_{L^p(\T)}.
$$
This proves the first estimate in \eqref{20260820eq16}.

\medskip

\noindent{\textit{\underline{Treatment of the case $k\geq n$}}}. We now prove the second estimate in \eqref{20260820eq16}. Observe that in this case, $\mathbb E_n({\bf d}_kh)=0$.  Hence, 
\begin{equation}\label{20260820eq20A}
A_{n,k}(h)^p
=\sum_{I\in\calD_n}|I|
\left(\sup_{z\in Q_I^{\textrm{up}}}|P[{\bf d}_kh](z)|^p \right).
\end{equation}

 Fix $I\in\calD_n$, and for each $J\in\calD_k$, let $c_J$ denote the center of $J$. Since ${\bf d}_kh$ has integral zero on every $J \in \calD_k$, we obtain, for every $z=re^{2\pi iy}\in Q_I^{\textrm{up}}$,  
\begin{align}
P[{\bf d}_kh](z)
&=\int_{\T} P_r(y-t){\bf d}_kh(t)\,dt \nonumber \\
&=\sum_{J\in\calD_k}\int_JP_r(y-t){\bf d}_kh(t)\,dt\nonumber\\
&=\sum_{J\in\calD_k}\int_J
\bigl(P_r(y-t)-P_r(y-c_J)\bigr){\bf d}_kh(t)\,dt.
\label{20260820eq20}
\end{align}
Since $z\in Q_I^{\textrm{up}}$, we have $1-r\simeq2^{-n}$. Taking derivatives of the Poisson kernel in \eqref{20260820eq14X} gives
\begin{equation}\label{20260820eq21}
|P_r'(u)|
\lesssim\frac{2^{-n}}
{\bigl(2^{-n}+d_{\T}(u,0)\bigr)^3}.
\end{equation}
On the other hand, for $t\in J$, we have $d_{\T}(t,c_J)\leq|J|=2^{-k} \le 2^{-n}$. Therefore, the mean value theorem and \eqref{20260820eq21} imply that
$$
|P_r(y-t)-P_r(y-c_J)|
\lesssim2^{-k}
\frac{2^{-n}}
{\bigl(2^{-n}+d_{\T}(y,t)\bigr)^3}.
$$
Substituting this estimate into \eqref{20260820eq20}, we obtain
\begin{equation}\label{20260820eq22}
|P[{\bf d}_kh](z)|
\lesssim2^{-k}\int_\T
\frac{2^{-n}|{\bf d}_kh(t)|}
{\bigl(2^{-n}+d_{\T}(y,t)\bigr)^3}\,dt.
\end{equation}
Next, take any $x \in I$. Since $y\in I$, we have $d_{\T}(x,y)\leq2^{-n}$. Decomposing $\T$ into dyadic annuli centered at $x$, \eqref{20260820eq22} gives
\begin{align*}
|P[{\bf d}_kh](z)|
&\lesssim2^{-k}\sum_{j=0}^\infty
\frac{2^{-n}}{(2^j2^{-n})^3}
\int_{\{d_{\T}(x,t)<C2^j2^{-n}\}}|{\bf d}_kh(t)|\,dt\\
&\lesssim2^{-k}\sum_{j=0}^\infty
\frac{2^{-n}}{(2^j2^{-n})^3} \cdot
\bigl(2^j2^{-n}\bigr)\calM({\bf d}_kh)(x)\\
&=2^{-(k-n)}\sum_{j=0}^\infty
2^{-2j}\calM({\bf d}_kh)(x)\\
&\lesssim2^{-(k-n)}\calM({\bf d}_kh)(x).
\end{align*}
Here, $\calM$ refers to the standard Hardy--Littlewood maximal operator on $\T$. Since the preceding estimate is uniform over $z\in Q_I^{\textrm{up}}$, it follows that
\begin{equation}\label{20260820eq22A}
\sup_{z\in Q_I^{\textrm{up}}}|P[{\bf d}_kh](z)|
\lesssim2^{-(k-n)}\calM({\bf d}_kh)(x),
\qquad x\in I, 
\end{equation}
which further gives
$$
|I|\sup_{z\in Q_I^{\textrm{up}}}|P[{\bf d}_kh](z)|^p
\lesssim2^{-p(k-n)}
\int_I|\calM({\bf d}_kh)(x)|^p\,dx.
$$
Summing over $I\in\calD_n$ and using \eqref{20260820eq20A} gives
\begin{align*}
A_{n,k}(h)^p
&\lesssim2^{-p(k-n)}
\sum_{I\in\calD_n}\int_I|\calM({\bf d}_kh)(x)|^p\,dx\\
&=2^{-p(k-n)}
\|\calM({\bf d}_kh)\|_{L^p(\T)}^p \\
&\lesssim_p2^{-p(k-n)}\|{\bf d}_kh\|_{L^p(\T)},
\end{align*}
where in the last step, we used the fact that $\calM$ maps $L^p(\T)$ boundedly into itself. This proves the second estimate in \eqref{20260820eq16}.

\medskip 

We now return to the proof of the main estimate \eqref{20260820eq14Y}. For $n\geq0$, define
\begin{equation}\label{20260820eq23}
A_n:=\left(
\sum_{I\in\calD_n}|I|
\sup_{z\in Q_I^{\textrm{up}}}
|P[h](z)-\mathcal W_{\calD}h(z)|^p
\right)^{1/p}.
\end{equation}
For $z=re^{2\pi it}\in Q_I^{\textrm{up}}$, we have
$$
P[h](z)-\mathcal W_{\calD}h(z)
=\sum_{k=0}^\infty
\left(P[{\bf d}_kh](z)-\mathbb E_n({\bf d}_kh)(t)\right).
$$
Thus, by generalized Minkowski's inequality and \eqref{20260820eq16}, we have
\begin{align}
A_n
&=\left(\sum_{I\in\calD_n}|I|
\sup_{z \in Q_I^{\textrm{up}}}
\left|\sum_{k=0}^\infty
\left(P[{\bf d}_kh](z)-\mathbb E_n({\bf d}_kh)(t)\right)\right|^p
\right)^{1/p}\nonumber\\
&\leq\left(\sum_{I\in\calD_n}|I|
\left(\sum_{k=0}^\infty
\sup_{z \in Q_I^{\textrm{up}}}
\left|P[{\bf d}_kh](z)-\mathbb E_n({\bf d}_kh)(t)\right|\right)^p
\right)^{1/p}\nonumber\\
&\leq\sum_{k=0}^\infty
\left(\sum_{I\in\calD_n}|I|
\sup_{z\in Q_I^{\textrm{up}}}
\left|P[{\bf d}_kh](z)-\mathbb E_n({\bf d}_kh)(t)\right|^p
\right)^{1/p}=\sum_{k=0}^\infty A_{n,k}(h)\nonumber\\
&\lesssim_p
\sum_{0 \le k<n}2^{-(n-k)/p}\|{\bf d}_kh\|_{L^p(\T)}
+\sum_{k\geq n}2^{-(k-n)}\|{\bf d}_kh\|_{L^p(\T)}.
\label{20260820eq24}
\end{align}
Since both $\{2^{-m/p}\}_{m\geq1}$ and $
\{2^{-m}\}_{m\geq0}$ belong to $\ell^1$, by Young's inequality, \eqref{20260820eq24}, and \eqref{20260820eq02A}, we deduce that 
\begin{align}
\sum_{n=0}^\infty A_n^p \lesssim_p\sum_{k=0}^\infty
\|{\bf d}_kh\|_{L^p(\T)}^p \lesssim_p\|h\|_{B_{p,p}^0(\T)}^p.
\label{20260820eq25}
\end{align}
Finally, for every $I\in\calD$,
$$
\mu(Q_I^{\textrm{up}})
\leq\mu(Q_I)
\leq \left\|\mu\right\|_{\mathcal{CM}, \calD} |I|.
$$
This together with \eqref{20260820eq23} and \eqref{20260820eq25} yields
\begin{align*}
\int_\D|P[h](z)-\mathcal W_{\calD}h(z)|^p\,d\mu(z)
&=\sum_{n=0}^\infty\sum_{I\in\calD_n}
\int_{Q_I^{\textrm{up}}}
|P[h](z)-\mathcal W_{\calD}h(z)|^p\,d\mu(z)\\
&\leq\sum_{n=0}^\infty\sum_{I\in\calD_n}
\mu(Q_I^{\textrm{up}})
\sup_{z\in Q_I^{\textrm{up}}}
|P[h](z)-\mathcal W_{\calD}h(z)|^p\\
&\leq\left\|\mu\right\|_{\mathcal{CM}, \calD}  \sum_{n=0}^\infty A_n^p\\
&\lesssim_p\left\|\mu\right\|_{\mathcal{CM}, \calD} 
\|h\|_{B_{p,p}^0(\T)}^p.
\end{align*}
This proves \eqref{20260820eq14Y}. The extension to every $h\in B_{p,p}^0(\T)$ follows by the density of Haar polynomials.
\end{proof}

We can now complete the second reduction.

\begin{prop}\label{20260820prop01}
Let $1<p<\infty$, and let $\mu$ be a positive Borel measure on $\D$. Then the harmonic embedding \eqref{20260819eq01} holds if and only if the Whitney embedding \eqref{20260820eq08} holds. Moreover,
\begin{equation}\label{20260820eq26}
\widehat C_{best}\simeq_p\widetilde C_{best}.
\end{equation}
\end{prop}

\begin{proof}
Suppose first that the harmonic embedding \eqref{20260819eq01} holds. By \eqref{20260820eq11},
$$
\left\|\mu\right\|_{\mathcal{CM},\calD}
\lesssim_p\widetilde C_{best}.
$$
For every Haar polynomial $h$, Theorem~\ref{20260820thm01} gives
\begin{align*}
\int_\D|\mathcal W_{\calD}h(z)|^p\,d\mu(z)
&\lesssim_p\int_\D|P[h](z)|^p\,d\mu(z)
+\int_\D|P[h](z)-\mathcal W_{\calD}h(z)|^p\,d\mu(z)\\
&\lesssim_p\left(\widetilde C_{best}
+\left\|\mu\right\|_{\mathcal{CM},\calD}\right)
\|h\|_{B_{p,p}^0(\T)}^p\\
&\lesssim_p\widetilde C_{best}\|h\|_{B_{p,p}^0(\T)}^p.
\end{align*}
Thus, the Whitney embedding \eqref{20260820eq08} holds and
\begin{equation}\label{20260820eq27}
\widehat C_{best}\lesssim_p\widetilde C_{best}.
\end{equation}

Conversely, suppose that the Whitney embedding \eqref{20260820eq08} holds. By \eqref{20260820eq12},
$$
\left\|\mu\right\|_{\mathcal{CM},\calD}
\lesssim_p\widehat C_{best}.
$$
For every Haar polynomial $h$, using Theorem~\ref{20260820thm01} again, we have 
\begin{align*}
\int_\D|P[h](z)|^p\,d\mu(z)
&\lesssim_p\int_\D|\mathcal W_{\calD}h(z)|^p\,d\mu(z)
+\int_\D|P[h](z)-\mathcal W_{\calD}h(z)|^p\,d\mu(z)\\
&\lesssim_p\left(\widehat C_{best}
+\left\|\mu\right\|_{\mathcal{CM},\calD}\right)
\|h\|_{B_{p,p}^0(\T)}^p\\
&\lesssim_p\widehat C_{best}\|h\|_{B_{p,p}^0(\T)}^p.
\end{align*}
Let $h\in B_{p,p}^0(\T)$, and choose Haar polynomials $\{h_m\}_{m\geq1}$ such that
$$
\|h_m-h\|_{B_{p,p}^0(\T)}\longrightarrow0.
$$
Since convergence in $B_{p,p}^0(\T)$ implies convergence in $\mathcal D'(\T)$ and the Poisson kernel is smooth, we have
$$
P[h_m](z)\longrightarrow P[h](z),\qquad z\in\D.
$$
It follows from Fatou's lemma that
\begin{align*}
\int_\D|P[h](z)|^p\,d\mu(z)
&\leq\liminf_{m\to\infty}\int_\D|P[h_m](z)|^p\,d\mu(z)\\
&\lesssim_p\widehat C_{best}\lim_{m\to\infty}
\|h_m\|_{B_{p,p}^0(\T)}^p\\
&=\widehat C_{best}\|h\|_{B_{p,p}^0(\T)}^p.
\end{align*}
Thus, the harmonic embedding \eqref{20260819eq01} holds and
\begin{equation}\label{20260820eq28}
\widetilde C_{best}\lesssim_p\widehat C_{best}.
\end{equation}
Combining \eqref{20260820eq27} and \eqref{20260820eq28} proves \eqref{20260820eq26}.
\end{proof}

In view of Proposition~\ref{20260820prop01}, the harmonic embedding problem \eqref{Harmonicembedding} has now been reduced to the Whitney embedding problem \eqref{20260820eq08}, or equivalently, the
tree estimate \eqref{20260820eq09}. The next step is to characterize this inequality in terms of the quantities $\calC_{p,\calD}(\mu)$ and $\calH_{p,\calD}(\mu)$ appearing in Theorem~\ref{20260813thmmain}.

\bigskip 

\section{Weighted Hardy inequalities on dyadic trees} \label{Sec04}

Observe that the tree estimate \eqref{20260820eq09} involves signed sums over the dyadic ancestors of a fixed interval. In this section, we recall the diagonal form of the weighted Hardy inequality from \cite[Theorem~3]{ArcozziRochbergSawyer2002} and record the two specializations needed for the packing energy and the Haar energy.

\subsection{A weighted Hardy inequality by Arcozzi, Rochberg, and Sawyer}\label{20260825subsec01}

For $1<t<\infty$, write $t'=t/(t-1)$. Given a weight
$\rho:\calD\to(0,\infty)$, define
$$
\mathfrak H_t(\mu,\rho)
:=\sup_{\substack{K\in\calD\\\mu(Q_K)>0}}
\frac{1}{\mu(Q_K)}
\left(\sum_{\substack{I\in\calD\\I\subseteq K}}
\mu(Q_I)^{t'}\rho(I)^{1-t'}\right).
$$

We record the following diagonal form of the weighted Hardy inequality of Arcozzi, Rochberg, and Sawyer on the upper Carleson tents from
\cite[Theorem~3]{ArcozziRochbergSawyer2002}.

\begin{prop}\label{20260825thm01}
Let $1<t<\infty$, let $\mu$ be a finite positive Borel measure on
$\D$, and let $\rho:\calD\to(0,\infty)$. Let
$N_t(\mu,\rho)$ be the least constant for which
\begin{equation}\label{20260825eq03}
\sum_{J\in\calD}
\left|\sum_{\substack{I\in\calD\\I\supseteq J}}b_I\right|^t
\mu(Q_J^{\textrm{up}})
\leq N_t(\mu,\rho)\sum_{I\in\calD}|b_I|^t\rho(I)
\end{equation}
holds for all finitely supported scalar sequences
$\{b_I\}_{I\in\calD}$. Then
$$
N_t(\mu,\rho)<\infty \qquad \textrm{if and only if} \qquad 
\mathfrak H_t(\mu,\rho)<\infty.
$$
Moreover,
\begin{equation}\label{20260825eq04}
N_t(\mu,\rho)\simeq_t\mathfrak H_t(\mu,\rho)^{t-1}.
\end{equation}
\end{prop}

\begin{proof}
Apply \cite[Theorem~3]{ArcozziRochbergSawyer2002} with $p=q=t$ there
and assign to each $J\in\calD$ the mass $\mu(Q_J^{\textrm{up}})$.
Thus, condition \cite[(6)]{ArcozziRochbergSawyer2002} becomes
precisely $\mathfrak H_t(\mu,\rho)<\infty$, while
\cite[(5)]{ArcozziRochbergSawyer2002} becomes
\eqref{20260825eq03} after taking the $t$-th power. This finishes the proof of the first part of the proposition. 

It remains to compare the optimal constants, which follows from a scaling argument. First note that the quantitative assertion in
\cite[Theorem~3]{ArcozziRochbergSawyer2002} gives
$N_t(\mu,\rho)\simeq_t1$ whenever $\mathfrak H_t(\mu,\rho)=1$. For every $\lambda>0$, note that
$$
N_t(\lambda\mu,\rho)=\lambda N_t(\mu,\rho)
\qquad\textrm{and}\qquad
\mathfrak H_t(\lambda\mu,\rho)
=\lambda^{t'-1}\mathfrak H_t(\mu,\rho).
$$
Assume that $\mathfrak H_t(\mu,\rho)>0$ and set
$$
\widetilde\mu:=\mathfrak H_t(\mu,\rho)^{-(t-1)}\mu.
$$
Since $(t-1)(t'-1)=1$, it follows that
$$
\mathfrak H_t(\widetilde\mu,\rho)=1.
$$
Applying the normalized estimate in \cite[Theorem~3]{ArcozziRochbergSawyer2002} to $\widetilde\mu$ gives
$$
\mathfrak H_t(\mu,\rho)^{-(t-1)}N_t(\mu,\rho)
=N_t(\widetilde\mu,\rho)\simeq_t1,
$$
which proves \eqref{20260825eq04}. The case
$\mathfrak H_t(\mu,\rho)=0$ is immediate.
\end{proof}

\subsection{Two consequences of the weighted Hardy inequality}\label{20260825subsec02}

We now record two consequences of Proposition~\ref{20260825thm01} that will play an important role in the proof of Theorem~\ref{20260813thmmain}.

\begin{cor}\label{20260825cor01}
Let $2<p<\infty$ and  $\mu$ be a finite positive Borel measure on
$\D$. Then
\begin{equation}\label{20260825eq09}
\sum_{J\in\calD}
\left(\sum_{\substack{I\in\calD\\I\supseteq J}}|a_I|^2\right)^{p/2}
\mu(Q_J^{\textrm{up}})
\leq C\sum_{I\in\calD}|a_I|^p|I|
\end{equation}
for every finitely supported scalar sequence $\{a_I\}_{I\in\calD}$ if and only if
$$
\calC_{p,\calD}(\mu)<\infty.
$$
Moreover, the optimal constant in \eqref{20260825eq09} is comparable to $\calC_{p,\calD}(\mu)^{\frac{p-2}{2}}$, up to a constant depending only on $p$, where $\calC_{p,\calD}(\mu)$ is the packing energy defined in Definition~\ref{20260825defn01}.
\end{cor}

\begin{proof}
Apply Proposition~\ref{20260825thm01} with
$$
t=\frac p2,\qquad t'=\frac{p}{p-2},\qquad \rho(I)=|I|,
$$
and replace $b_I$ by $|a_I|^2$. In this case,
\begin{align*}
\mathfrak H_{p/2}(\mu,\rho)
&=\sup_{\substack{K\in\calD\\\mu(Q_K)>0}}
\frac{1}{\mu(Q_K)}
\left(\sum_{\substack{I\in\calD\\I\subseteq K}}
\mu(Q_I)^{\frac{p}{p-2}}|I|^{-\frac{2}{p-2}} \right)\\
&=\sup_{\substack{K\in\calD\\\mu(Q_K)>0}}
\frac{1}{\mu(Q_K)}
\left(\sum_{\substack{I\in\calD\\I\subseteq K}}
\left(\frac{\mu(Q_I)}{|I|}\right)^{\frac{p}{p-2}}|I| \right)\\
&=\calC_{p,\calD}(\mu).
\end{align*}
Since $t-1=(p-2)/2$, \eqref{20260825eq09} follows from \eqref{20260825eq04}.
\end{proof}

\begin{cor}\label{20260825cor02}
Let $2<p<\infty$ and  $\mu$ be a finite positive Borel measure on
$\D$. Then
\begin{equation}\label{20260825eq10}
\sum_{J\in\calD}\left|\sum_{\substack{I\in\calD\\I\supseteq J}}a_I\frac{\mu(Q_{I_+})-\mu(Q_{I_-})}{\mu(Q_I)}\right|^p\mu(Q_J^{\textrm{up}})
\leq C\sum_{I\in\calD}|a_I|^p|I|
\end{equation}
for every finitely supported scalar sequence $\{a_I\}_{I\in\calD}$ if and only if
$$
\calH_{p,\calD}(\mu)<\infty.
$$
Moreover, the optimal constant in \eqref{20260825eq10} is comparable to $\calH_{p,\calD}(\mu)^{p-1}$, up to a constant depending only on $p$, where $\calH_{p,\calD}(\mu)$ is the Haar energy defined in Definition~\ref{20260825defn01}. Here and throughout, we adopt the convention that
$$
\frac{\mu(Q_{I_+})-\mu(Q_{I_-})}{\mu(Q_I)}=0
$$
whenever $\mu(Q_I)=0$.
\end{cor}

\begin{proof}
It is enough to consider those $I\in\calD$ for which
$\mu(Q_{I_+})-\mu(Q_{I_-})\neq 0$. Apply Proposition~\ref{20260825thm01} with $t=p$ and
$$
\rho(I):=|I|\left|
\frac{\mu(Q_{I_+})-\mu(Q_{I_-})}{\mu(Q_I)}
\right|^{-p},
$$
and replace $b_I$ by
$$
a_I\frac{\mu(Q_{I_+})-\mu(Q_{I_-})}{\mu(Q_I)}.
$$
Since $t'=p/(p-1)$, a direct computation gives
\begin{align*}
\mu(Q_I)^{t'}\rho(I)^{1-t'}
&=\mu(Q_I)^{\frac{p}{p-1}}|I|^{-\frac{1}{p-1}}
\left|
\frac{\mu(Q_{I_+})-\mu(Q_{I_-})}{\mu(Q_I)}
\right|^{\frac{p}{p-1}}\\
&=\left|\mu(Q_{I_+})-\mu(Q_{I_-})\right|^{\frac{p}{p-1}}
|I|^{-\frac{1}{p-1}}\\
&=\left(
\frac{\left|\mu(Q_{I_+})-\mu(Q_{I_-})\right|}{|I|}
\right)^{\frac{p}{p-1}}|I|.
\end{align*}
Therefore,
$$
\mathfrak H_p(\mu,\rho)=\calH_{p,\calD}(\mu).
$$
Moreover,
$$
\left|a_I\frac{\mu(Q_{I_+})-\mu(Q_{I_-})}{\mu(Q_I)}
\right|^p\rho(I)
=|a_I|^p|I|.
$$
Since $t-1=p-1$, \eqref{20260825eq10} follows from Theorem~\ref{20260825thm01}.
\end{proof}

The estimates \eqref{20260825eq09} and \eqref{20260825eq10} will be used in the next section to prove the two-energy characterization of the tree estimate \eqref{20260820eq09}.

\bigskip 

\section{Proof of Theorem \ref{20260813thmmain}}\label{Sec05}

We now prove Theorem~\ref{20260813thmmain}. By Proposition~\ref{20260819prop01} and Proposition~\ref{20260820prop01}, the Dirichlet embedding \eqref{Dirichletembedding01}, the harmonic embedding \eqref{20260819eq01}, and the Whitney embedding \eqref{20260820eq08} are equivalent. In particular, the optimal constants in these embeddings enjoy
$$
C_{best}\simeq_p\widetilde C_{best}\simeq_p\widehat C_{best}.
$$
It therefore remains to characterize the Whitney embedding \eqref{20260820eq08}, or equivalently, the tree estimate \eqref{20260820eq09}: 
\begin{equation} \label{20260825eq12}
\sum_{J\in\calD}
\left|\widehat h(0)+\sum_{\substack{I\in\calD \\ I\supsetneq J}}a_I \left(\left.h_I\right|_J\right)\right|^p
\mu(Q_J^{\textrm{up}})
\le \breve C_{best} \left(|\widehat h(0)|^p+\sum_{I\in\calD}|a_I|^p|I| \right), 
\end{equation}
where $\breve C_{best}$ denotes the optimal constant in
\eqref{20260825eq12}. By \eqref{20260820eq02} and \eqref{20260820eq07}, it satisfies
$$
\breve C_{best} \simeq_p C_{best}. 
$$
Consequently, to prove Theorem \ref{20260813thmmain}, it suffices to prove
\begin{equation}\label{20260825eq11}
\breve C_{best} \simeq_p
\calC_{p,\calD}(\mu)^{\frac{p-2}{2}}
+\calH_{p,\calD}(\mu)^{p-1}.
\end{equation}

We need the following lemma first. 

\begin{lem}\label{20260827lem01}
Let $2<p<\infty$, and suppose that $\calC_{p,\calD}(\mu)<\infty$, where  we recall that $\calC_{p,\calD}(\mu)$ denotes the packing energy of $\mu$ defined in Definition~\ref{20260825defn01}. For each $I\in\calD$, set
$$
m_I:=
\begin{cases}
\displaystyle\frac{\mu(Q_{I_+})-\mu(Q_{I_-})}{\mu(Q_I)},&\mu(Q_I)>0,\\
0,&\mu(Q_I)=0.
\end{cases}
$$
Then, for every finitely supported sequence $\{a_I\}_{I\in\calD}$,
\begin{align}\label{20260827eq01}
&\sum_{J\in\calD}\left|\sum_{I\in\calD}a_I\left(\one_{\{J\subseteq I_+\}}-\one_{\{J\subseteq I_-\}}-m_I\one_{\{J\subseteq I\}}\right)\right|^p\mu(Q_J^{\textrm{up}})\nonumber\\
&\qquad\lesssim_p\calC_{p,\calD}(\mu)^{\frac{p-2}{2}}\sum_{I\in\calD}|a_I|^p|I|.
\end{align}
\end{lem}

\begin{proof}
For each $J \in \calD$, define $\nu(\{J\}):=\mu(Q_J^{\textrm{up}})$, and also write
$\mathcal S(I):=\{J\in\calD:J\subseteq I\}$. 
Therefore, 
$$
\nu(\mathcal S(I))=\sum_{J \subseteq I} \nu(\{J\})=\sum_{J\subseteq I}\mu(Q_J^{\textrm{up}})=\mu(Q_I).
$$
In particular, $\nu(\calD)=\mu(\D)<+\infty$, and hence $(\calD, \nu)$ becomes a finite measure space. 
For each $n\geq0$, let $\mathcal F_n$ be the $\sigma$-algebra whose atoms are the singletons 
$$
\{J\},\qquad J\in\calD_k,\quad 0\leq k<n,
$$
together with the sets
$$
\mathcal S(I),\qquad I\in\calD_n.
$$
Thus, when passing from $\mathcal F_n$ to $\mathcal F_{n+1}$, each atom $\mathcal S(I)$, $I\in\calD_n$, splits into
\begin{equation} \label{20260827eq20}
\{I\},\qquad \mathcal S(I_+),\qquad \mathcal S(I_-).
\end{equation} 
For each $n\geq0$, define
$$
d_n(J):=\sum_{I\in\calD_n}a_I
\left(
\one_{\{J\subseteq I_+\}}
-\one_{\{J\subseteq I_-\}}
-m_I\one_{\{J\subseteq I\}}
\right),
\qquad J\in\calD.
$$
Since $\{a_I\}_{I\in\calD}$ is finitely supported, only finitely many of the functions $d_n$ are nonzero.

\medskip

\noindent\emph{\underline{Claim:}} The sequence $\{d_n\}_{n\geq0}$ is a martingale-difference sequence on $(\calD,\nu)$ with respect to the filtration $\{\mathcal F_n\}_{n\geq0}$; that is, $d_n$ is $\mathcal F_{n+1}$-measurable and
$$
\mathbb E_\nu(d_n\mid\mathcal F_n)=0
$$
for every $n\geq0$.

\smallskip

\noindent\emph{\underline{Proof of the claim:}}
Fix $n\geq0$. If $J\in\calD_k$ for some $k<n$, then $J$ is not contained in any interval belonging to $\calD_n$, and hence $d_n(J)=0$.  Now fix $I\in\calD_n$. By its definition, $d_n$ is constant on each of the three sets in \eqref{20260827eq20}, taking the values
$$
-a_Im_I,\qquad a_I(1-m_I),\qquad -a_I(1+m_I),
$$
respectively. It follows that $d_n$ is $\mathcal F_{n+1}$-measurable. 

It remains to verify that $d_n$ has integral zero on every atom of $\mathcal F_n$. This is immediate for the singleton atoms $\{J\}$ with $J\in\calD_k$ and $k<n$, since $d_n(J)=0$. For an atom $\mathcal S(I)$ with $I\in\calD_n$, we have
\begin{align*}
\int_{\mathcal S(I)}d_n\,d\nu
&=\sum_{J\subseteq I}d_n(J)\mu(Q_J^{\textrm{up}})\\
&=a_I\left(
\sum_{J\subseteq I_+}\mu(Q_J^{\textrm{up}})
-\sum_{J\subseteq I_-}\mu(Q_J^{\textrm{up}})
-m_I\sum_{J\subseteq I}\mu(Q_J^{\textrm{up}})
\right)\\
&=a_I\left(\mu(Q_{I_+})-\mu(Q_{I_-})-m_I\mu(Q_I)\right)=0.
\end{align*}
Therefore, $d_n$ has integral zero on every atom of $\mathcal F_n$, and hence $\mathbb E_\nu(d_n\mid\mathcal F_n)=0$. This proves the claim.

\medskip

By the standard $L^p$ martingale square-function inequality and the claim,
\begin{equation} \label{20260827eq01A}
\sum_{J\in\calD}\left|
\sum_{n=0}^{\infty}d_n(J)
\right|^p \mu(Q_J^{\textrm{up}}) \lesssim_p
\sum_{J\in\calD}
\left(\sum_{n=0}^{\infty}|d_n(J)|^2
\right)^{p/2} \mu(Q_J^{\textrm{up}}).
\end{equation}
For each fixed $J\in\calD$ and each $n\geq0$, there is at most one interval $I\in\calD_n$ such that $I\supseteq J$. Since $|m_I|\leq1$, it follows that
$$
|d_n(J)|
\leq
2\left(
\sum_{\substack{I\in\calD_n\\I\supseteq J}}|a_I|^2
\right)^{1/2}.
$$
Consequently,
\begin{equation} \label{20260827eq01B}
\left(\sum_{n=0}^{\infty}|d_n(J)|^2
\right)^{1/2} \leq 2\left(
\sum_{\substack{I\in\calD\\I\supseteq J}}|a_I|^2
\right)^{1/2}.
\end{equation} 
Therefore, by \eqref{20260827eq01A}, \eqref{20260827eq01B}, and Corollary~\ref{20260825cor01}, we conclude that 
\begin{align*}
\sum_{J\in\calD}
\left|
\sum_{n=0}^{\infty}d_n(J)
\right|^p
\mu(Q_J^{\textrm{up}})
&\lesssim_p
\sum_{J\in\calD}
\left(
\sum_{\substack{I\in\calD\\I\supseteq J}}|a_I|^2
\right)^{p/2}
\mu(Q_J^{\textrm{up}})\\
&\lesssim_p
\calC_{p,\calD}(\mu)^{\frac{p-2}{2}}
\sum_{I\in\calD}|a_I|^p|I|,
\end{align*}
which gives the desired estimate \eqref{20260827eq01} since 
$$
\sum_{n=0}^{\infty}d_n(J)=
\sum_{I\in\calD}a_I
\left(\one_{\{J\subseteq I_+\}}
-\one_{\{J\subseteq I_-\}}
-m_I\one_{\{J\subseteq I\}}
\right).
$$
\end{proof}

\vspace{0.1cm}

We are now ready to prove the main result of the paper.

\begin{proof}[Proof of Theorem~\ref{20260813thmmain}] We divide the proof into three steps.

\vspace{0.1cm} 

\noindent{\bf Step I: Necessity of the packing energy $\calC_{p,\calD}(\mu)$}.
Assume that \eqref{20260825eq12} holds. Take any finitely supported scalar sequence $\{a_I\}_{I\in\calD}$, and let $\{\varepsilon_I\}_{I\in\calD}$ be independent Rademacher variables. Apply \eqref{20260825eq12} with $\widehat h(0)=0$ and with $a_I$ replaced by $\varepsilon_Ia_I$. Since $h_I|_J\in\{-1,1\}$ whenever $J\subsetneq I$, Khinchine's inequality gives
$$
\mathbb E_{\varepsilon}\left|\sum_{\substack{I\in\calD\\I\supsetneq J}}\varepsilon_Ia_I\left(h_I|_J\right)\right|^p
\simeq_p\left(\sum_{\substack{I\in\calD\\I\supsetneq J}}|a_I|^2\right)^{p/2}.
$$
Integrating \eqref{20260825eq12} with respect to the Rademacher variables, we obtain
\begin{equation}\label{20260825eq13}
\sum_{J\in\calD}\left(\sum_{\substack{I\in\calD\\I\supsetneq J}}|a_I|^2\right)^{p/2}\mu(Q_J^{\textrm{up}})
\lesssim_p\breve C_{best}\sum_{I\in\calD}|a_I|^p|I|.
\end{equation}
Observe that Corollary~\ref{20260825cor01} cannot be applied directly to \eqref{20260825eq13}, since the inner sum in \eqref{20260825eq09} is taken over $I\supseteq J$, whereas the inner sum in \eqref{20260825eq13} is taken only over $I\supsetneq J$. Our next goal is to show that \eqref{20260825eq13} remains valid when $I\supsetneq J$ is replaced by $I\supseteq J$.

Recall that every dyadic interval other than $\T$ is either the $+$ child or the $-$ child of its unique dyadic parent. Define
$$
\calD^+:=\{I_+:I\in\calD\},
\qquad
\calD^-:=\{I_-:I\in\calD\}.
$$
Thus,
$$
\calD\setminus\{\T\}=\calD^+\sqcup\calD^-.
$$
If $I=K_+\supseteq J$, then $K\supsetneq J$. Consequently,
$$
\sum_{\substack{I\in\calD^+\\I\supseteq J}}|a_I|^2
\leq\sum_{K\supsetneq J}|a_{K_+}|^2,
$$
while
$$
\sum_{K\in\calD}|a_{K_+}|^p|K|
=2\sum_{I\in\calD^+}|a_I|^p|I|.
$$
Applying \eqref{20260825eq13} to the sequence $\{a_{K_+}\}_{K\in\calD}$ gives
\begin{equation} \label{20260825eq100}
\sum_{J\in\calD}\left(\sum_{\substack{I\in\calD^+\\I\supseteq J}}|a_I|^2\right)^{p/2}\mu(Q_J^{\textrm{up}})
\lesssim_p\breve C_{best}\sum_{I\in\calD^+}|a_I|^p|I|.
\end{equation}
Similarly, applying \eqref{20260825eq13} to $\{a_{K_-}\}_{K\in\calD}$ gives
\begin{equation} \label{20260825eq101}
\sum_{J\in\calD}\left(\sum_{\substack{I\in\calD^-\\I\supseteq J}}|a_I|^2\right)^{p/2}\mu(Q_J^{\textrm{up}})
\lesssim_p\breve C_{best}\sum_{I\in\calD^-}|a_I|^p|I|.
\end{equation}
Taking $\widehat h(0)=a_\T$ and all Haar coefficients equal to zero in \eqref{20260825eq12}, we also obtain
\begin{equation} \label{20260825eq102}
|a_\T|^p\mu(Q_\T)\leq\breve C_{best}|a_\T|^p.
\end{equation} 
Finally, combining the estimates \eqref{20260825eq100}, \eqref{20260825eq101}, and \eqref{20260825eq102}, we deduce that
$$
\sum_{J\in\calD}\left(\sum_{I\supseteq J}|a_I|^2\right)^{p/2}\mu(Q_J^{\textrm{up}})
\lesssim_p\breve C_{best}\sum_{I\in\calD}|a_I|^p|I|, 
$$
which, together with Corollary~\ref{20260825cor01}, gives
\begin{equation}\label{20260825eq15}
\calC_{p,\calD}(\mu)^{\frac{p-2}{2}}\lesssim_p\breve C_{best}.
\end{equation}

\medskip

\noindent{\bf Step II: Necessity of the Haar energy $\calH_{p,\calD}(\mu)$}.
We continue to assume that \eqref{20260825eq12} holds. Setting $\widehat h(0)=0$ in \eqref{20260825eq12}, we obtain
\begin{equation} \label{20260826eq01}
\sum_{J\in\calD}\left|\sum_{\substack{I\in\calD\\I\supsetneq J}}a_I\left(h_I|_J\right)\right|^p\mu(Q_J^{\textrm{up}})
\leq\breve C_{best}\sum_{I\in\calD}|a_I|^p|I|
\end{equation} 
for every finitely supported sequence $\{a_I\}_{I\in\calD}$. We next derive the dual form of \eqref{20260826eq01}. More precisely, define
$$
(Ta)_J:=\sum_{\substack{I\in\calD\\I\supsetneq J}}a_I\left(h_I|_J\right),\qquad J\in\calD.
$$
This allows one to rewrite \eqref{20260826eq01} as
\begin{equation}  \label{20260826eq02}
\left(\sum_{J\in\calD}|(Ta)_J|^p\mu(Q_J^{\textrm{up}})\right)^{1/p}
\leq\breve C_{best}^{1/p}\left(\sum_{I\in\calD}|a_I|^p|I|\right)^{1/p}.
\end{equation} 
Let $\{b_J\}_{J\in\calD}$ be a finitely supported scalar sequence. By H\"older's inequality and \eqref{20260826eq02},
\begin{align} \label{20260826eq03}
\left|\sum_{J\in\calD}(Ta)_J\overline{b_J}\mu(Q_J^{\textrm{up}})\right|
&\leq\left(\sum_{J\in\calD}|(Ta)_J|^p\mu(Q_J^{\textrm{up}})\right)^{1/p}
\left(\sum_{J\in\calD}|b_J|^{p'}\mu(Q_J^{\textrm{up}})\right)^{1/p'} \nonumber \\
&\leq\breve C_{best}^{1/p}\left(\sum_{I\in\calD}|a_I|^p|I|\right)^{1/p}
\left(\sum_{J\in\calD}|b_J|^{p'}\mu(Q_J^{\textrm{up}})\right)^{1/p'}.
\end{align}
On the other hand, since both $\{a_I\}$ and $\{b_J\}$ are finitely supported, we have
\begin{align} \label{20260826eq04}
\sum_{J\in\calD}(Ta)_J\overline{b_J}\mu(Q_J^{\textrm{up}})
&=\sum_{J\in\calD}\sum_{\substack{I \in \calD \\ I\supsetneq J}}a_I\left(h_I|_J\right)\overline{b_J}\mu(Q_J^{\textrm{up}}) \nonumber \\
&=\sum_{I\in\calD}a_I \left(\sum_{\substack{J \in \calD \\ J\subsetneq I}}\left(h_I|_J\right)\overline{b_J}\mu(Q_J^{\textrm{up}}) \right).
\end{align}
For each $I\in\calD$, set
$$
B_I:=\sum_{\substack{J \in \calD \\ J\subsetneq I}}\left(h_I|_J\right)b_J\mu(Q_J^{\textrm{up}}).
$$
Thus, by \eqref{20260826eq03} and \eqref{20260826eq04}, 
$$
\left|\sum_{I\in\calD}a_I\overline{B_I}\right|=\left|\sum_{I\in\calD}\left(a_I |I|^{1/p} \right)(\overline{B_I|I|^{-1/p}})\right|
\leq\breve C_{best}^{1/p}\left(\sum_{I\in\calD}|a_I|^p|I|\right)^{1/p}
\left(\sum_{J\in\calD}|b_J|^{p'}\mu(Q_J^{\textrm{up}})\right)^{1/p'}.
$$
Taking the supremum over all finitely supported sequences $\{a_I\}$ satisfying $\sum_{I\in\calD}|a_I|^p|I|\le 1$,
we obtain
$$
\left(\sum_{I\in\calD}|B_I|^{p'}|I|^{1-p'}\right)^{1/p'}
\leq\breve C_{best}^{1/p}
\left(\sum_{J\in\calD}|b_J|^{p'}\mu(Q_J^{\textrm{up}})\right)^{1/p'}.
$$
Therefore, 
\begin{equation}\label{20260825eq16}
\sum_{I\in\calD}\left|\sum_{\substack{J\in\calD\\J\subsetneq I}}\left(h_I|_J\right)b_J\mu(Q_J^{\textrm{up}})\right|^{p'}|I|^{1-p'}
\leq\breve C_{best}^{p'/p}\sum_{J\in\calD}|b_J|^{p'}\mu(Q_J^{\textrm{up}})
\end{equation}
for every finitely supported scalar sequence $\{b_J\}_{J\in\calD}$.

\vspace{0.1cm}

Fix $K\in\calD$. For every positive integer $N$, apply \eqref{20260825eq16} to
$$
b_J=\one_{\{J\subseteq K,\ |J|\geq2^{-N}\}}, \qquad J \in \calD. 
$$
For every $I\subseteq K$, \eqref{20260821eq02} and
\eqref{20260820eq03} give
\begin{align*}
\lim_{N\to\infty}
\sum_{\substack{J\subsetneq I\\|J|\geq2^{-N}}}
\left(h_I|_J\right)\mu(Q_J^{\textrm{up}})
&=\sum_{J\subseteq I_+}\mu(Q_J^{\textrm{up}})
-\sum_{J\subseteq I_-}\mu(Q_J^{\textrm{up}})\\
&\qquad=
\mu(Q_{I_+})-\mu(Q_{I_-}).
\end{align*} 
Letting $N\to\infty$ in \eqref{20260825eq16} with the above choice of $b_J$, and applying Fatou's lemma, we deduce that for any $K \in \calD$, 
$$
\sum_{I\subseteq K}\left|\mu(Q_{I_+})-\mu(Q_{I_-})\right|^{p'}|I|^{1-p'}
\leq\breve C_{best}^{p'/p}\mu(Q_K).
$$
Dividing by $\mu(Q_K)$ and taking the supremum over all $K\in\calD$ such that $\mu(Q_K)>0$, the definition of the Haar energy (see Definition \ref{20260825defn01}) gives
$\calH_{p,\calD}(\mu)\leq\breve C_{best}^{p'/p}$. Since $p'/p=1/(p-1)$, it follows that
\begin{equation}\label{20260825eq17}
\calH_{p,\calD}(\mu)^{p-1}\leq\breve C_{best}.
\end{equation}

\medskip

\medskip

\noindent{\bf Step III: Sufficiency}.
Assume that
$$
\calC_{p,\calD}(\mu)^{\frac{p-2}{2}}+\calH_{p,\calD}(\mu)^{p-1}<+\infty.
$$
Our goal is to prove the following tree estimate
\begin{equation} \label{20260825eq12AB}
\sum_{J\in\calD}
\left|\widehat h(0)+\sum_{\substack{I\in\calD \\ I\supsetneq J}}a_I \left(\left.h_I\right|_J\right)\right|^p
\mu(Q_J^{\textrm{up}})
\lesssim_p \left(\calC_{p,\calD}(\mu)^{\frac{p-2}{2}}+\calH_{p,\calD}(\mu)^{p-1}\right) \left(|\widehat h(0)|^p+\sum_{I\in\calD}|a_I|^p|I| \right).
\end{equation}
If $\mu=0$, there is nothing to prove. Hence, assume that
$\mu(\D)>0$. Let $\widehat h(0)\in\C$ and let $\{a_I\}_{I\in\calD}$ be a finitely supported scalar sequence. For each $I\in\calD$, let $m_I$ be as in Lemma~\ref{20260827lem01}. For each $J\in\calD$, decompose
\begin{equation}\label{20260825eq18}
\widehat h(0)+\sum_{\substack{I\in\calD\\I\supsetneq J}}a_I\left(h_I|_J\right)
={\bf A}_1(J)+{\bf A}_2(J)+{\bf A}_3(J),
\end{equation}
where
\begin{align*}
{\bf A}_1(J)&:=\widehat h(0),\\
{\bf A}_2(J)&:=\sum_{I\in\calD}a_I\left(\one_{\{J\subseteq I_+\}}-\one_{\{J\subseteq I_-\}}-m_I\one_{\{J\subseteq I\}}\right),\\
{\bf A}_3(J)&:=\sum_{\substack{I\in\calD\\I\supseteq J}}a_Im_I.
\end{align*}
Indeed, \eqref{20260825eq18} follows from
$$
\one_{\{J\subseteq I_+\}}-\one_{\{J\subseteq I_-\}}
=
\begin{cases}
h_I|_J,& I\supsetneq J,\\
0,& \text{otherwise}.
\end{cases}
$$

\medskip 

\noindent{\textit{\underline{Treatment of ${\bf A}_1$}}}.
By Definition \ref{20260825defn01}, we have
\begin{align*}
\calC_{p, \calD}(\mu)
& \ge \frac{1}{\mu(\D)} \left[\sum_{I \in \calD} \left(\frac{\mu(Q_I)}{|I|}\right)^{\frac{p}{p-2}}|I| \right] \ge \frac{1}{\mu(\D)} \cdot \left(\frac{\mu(\D)}{|\T|}\right)^{\frac{p}{p-2}}|\T| \gtrsim_p \mu(\D)^{\frac{2}{p-2}}.
\end{align*}
Therefore, 
\begin{equation} \label{20260825eq19}
\sum_{J\in\calD}|\widehat h(0)|^p\mu(Q_J^{\textrm{up}})
=|\widehat h(0)|^p\mu(\D)
\leq\calC_{p,\calD}(\mu)^{\frac{p-2}{2}}|\widehat h(0)|^p.
\end{equation}

\vspace{0.1cm}

\noindent{\textit{\underline{Treatment of ${\bf A}_2$}}}.
By Lemma~\ref{20260827lem01},
\begin{align}\label{20260825eq20}
&\sum_{J\in\calD}\left|\sum_{I\in\calD}a_I\left(\one_{\{J\subseteq I_+\}}-\one_{\{J\subseteq I_-\}}-m_I\one_{\{J\subseteq I\}}\right)\right|^p\mu(Q_J^{\textrm{up}})\nonumber\\
&\qquad\lesssim_p\calC_{p,\calD}(\mu)^{\frac{p-2}{2}}\sum_{I\in\calD}|a_I|^p|I|.
\end{align}

\vspace{0.1cm}

\noindent{\textit{\underline{Treatment of ${\bf A}_3$}}}.
By Corollary~\ref{20260825cor02},
\begin{equation}\label{20260825eq21}
\sum_{J\in\calD}\left|\sum_{\substack{I\in\calD\\I\supseteq J}}a_Im_I\right|^p\mu(Q_J^{\textrm{up}})
\lesssim_p\calH_{p,\calD}(\mu)^{p-1}\sum_{I\in\calD}|a_I|^p|I|.
\end{equation}

\vspace{0.1cm}

Combining \eqref{20260825eq18}--\eqref{20260825eq21} yields the desired estimate \eqref{20260825eq12AB}. Finally, since $\breve C_{best}$ is the optimal constant in \eqref{20260825eq12}, it follows that
$$
\breve C_{best}\lesssim_p\calC_{p,\calD}(\mu)^{\frac{p-2}{2}}+\calH_{p,\calD}(\mu)^{p-1}.
$$
Together with \eqref{20260825eq15} and \eqref{20260825eq17}, this proves \eqref{20260825eq11} and completes the proof of Theorem~\ref{20260813thmmain}.
\end{proof}

\bigskip 

\section{Proof of Theorem~\ref{mainthm02}} \label{Sec06}

The purpose of this section is to show that the packing energy and the Haar energy capture different properties of the measure. We prove this by constructing finite atomic measures whose packing energies remain uniformly bounded while their Haar energies become arbitrarily large. We begin with the following lemma.

\begin{lem}\label{20260828lem01}
Let $2<p<\infty$, $N\geq2$, $I_0\in\calD$, and
$$
\calS_L(I_0):=\left\{I \in \calD: \; I \subseteq I_0, \; |I| \ge 2^{-L}{|I_0|} \right\}. 
$$
For every sufficiently small $a>0$, there exist a positive integer $L$ and positive sequence 
$$
\{D(I)\}_{I\in\mathcal S_L(I_0)}
$$
with the following properties:
\begin{enumerate}
\item $D(I_0)=a$, $D(I)\simeq a$ for every $I\in\mathcal S_L(I_0)$, and
$$
D(I)=\frac{D(I_+)+D(I_-)}{2}
$$
whenever $I_+,I_-\in\mathcal S_L(I_0)$;
\item
\begin{equation}\label{20260828eq01}
\sum_{I\in\mathcal S_L(I_0)}D(I)^{\frac{p}{p-2}}|I|
\simeq_p N^{-2}a|I_0|;
\end{equation}
\item
\begin{equation}\label{20260828eq02}
\sum_{\substack{I\in\mathcal S_L(I_0)\\|I|>2^{-L}|I_0|}}
|D(I_+)-D(I)|^{\frac{p}{p-1}}|I|
\simeq_p N^{-\frac{p-2}{p-1}}a|I_0|.
\end{equation}
\end{enumerate}
Moreover, for every $J\in\mathcal S_L(I_0)$,
\begin{equation}\label{20260828eq03}
\sum_{\substack{I\in\mathcal S_L(I_0)\\I\subseteq J}}
D(I)^{\frac{p}{p-2}}|I|
\lesssim_p N^{-2}D(J)|J|
\end{equation}
and
\begin{equation}\label{20260828eq04}
\sum_{\substack{I\in\mathcal S_L(I_0)\\I\subseteq J\\|I|>2^{-L}|I_0|}}
|D(I_+)-D(I)|^{\frac{p}{p-1}}|I|
\lesssim_p N^{-\frac{p-2}{p-1}}D(J)|J|.
\end{equation}
\end{lem}

\begin{proof}
Set
$$
L:=\left\lceil N^{-2}a^{-\frac{2}{p-2}}\right\rceil.
$$
Since $a$ is sufficiently small, we may assume that $L\geq2$. In particular,
\begin{equation}\label{20260828eq05}
L\simeq N^{-2}a^{-\frac{2}{p-2}}.
\end{equation}
Let $K$ be a sufficiently large fixed constant and set
$$
\eta:=\frac{a}{K\sqrt L}.
$$
We define $D(I)$ recursively for $I\in\mathcal S_L(I_0)$. Set $D(I_0)=a$ and declare $I_0$ is selected. Suppose that $D(I)$ has been defined for some $I\in\mathcal S_L(I_0)$ satisfying $|I|>2^{-L}|I_0|$. If $I$ is selected, define
\begin{equation}\label{20260828eq06}
D(I_+):=D(I)+\eta,\qquad D(I_-):=D(I)-\eta.
\end{equation}
Each child is selected if its value belongs to $[a/2,3a/2]$ and unselected otherwise. If $I$ is unselected, define
\begin{equation}\label{20260828eq07}
D(I_+):=D(I_-):=D(I)
\end{equation}
and therefore both children are unselected. Continuing in this manner defines $D(I)$ for every $I\in\mathcal S_L(I_0)$.

It follows immediately from \eqref{20260828eq06} and \eqref{20260828eq07} that
\begin{equation}\label{20260828eq08}
D(I)=\frac{D(I_+)+D(I_-)}{2}
\end{equation}
whenever $I\in\mathcal S_L(I_0)$ and $|I|>2^{-L}|I_0|$. Moreover, for any unselected arc $I$, its value $D(I)$ differs from one of the endpoints of $[a/2,3a/2]$ by at most $\eta$. Its value remains unchanged at all its dyadic offsprings. Since $\eta\leq a/K$, we may therefore assume
\begin{equation}\label{20260828eq09}
\frac{a}{4}\leq D(I)\leq2a,\qquad I\in\mathcal S_L(I_0)
\end{equation}
by choosing $K$ sufficiently large. 

\vspace{0.1cm}

We next estimate the total length of the selected intervals. For every $x\in I_0$ and $0\leq k\leq L$, let $I_k(x)$ be the dyadic subinterval of $I_0$ containing $x$ such that $|I_k(x)|=2^{-k}|I_0|$. For $1\leq j\leq L$, define
$$
r_j(x):=
\begin{cases}
1,&I_j(x)=[I_{j-1}(x)]_+;\\
\\
-1,&I_j(x)=[I_{j-1}(x)]_-.
\end{cases}
$$
Thus, $\{r_j\}_{j=1}^L$ is the Rademacher system on the arc $I_0$ with respect to normalized Lebesgue measure. Fix $1\leq k\leq L$. If the intervals
$$
I_0(x)=I_0, \; I_1(x),\; \ldots,I_{k-1}(x)
$$
are all selected, then applying \eqref{20260828eq06} repeatedly, we conclude that 
\begin{align*}
D(I_k(x))
&=D(I_{k-1}(x))+\eta r_k(x)\\
&=D(I_{k-2}(x))+\eta\bigl(r_{k-1}(x)+r_k(x)\bigr)\\
&=\cdots\\
&=D(I_0)+\eta\sum_{j=1}^kr_j(x)
=a+\eta\sum_{j=1}^kr_j(x).
\end{align*}
Consequently, if the interval containing $x$ is unselected at or before generation $L$, then
$$
\max_{1\leq k\leq L}\left|\sum_{j=1}^kr_j(x)\right|
\geq\frac{a}{2\eta}
=\frac{K\sqrt L}{2}.
$$
Therefore, since the functions $\{r_j\}_{1 \le j \le L}$ are independent and have mean zero, Kolmogorov's maximal inequality, together with their orthogonality, gives
\begin{equation}\label{20260828eq10}
\left|\left\{x\in I_0:
\max_{1\leq k\leq L}\left|\sum_{j=1}^kr_j(x)\right|
\geq\frac{K\sqrt L}{2}\right\}\right| \leq\frac{4}{K^2L}
\int_{I_0}\left|\sum_{j=1}^Lr_j(x)\right|^2\,dx =\frac{4}{K^2}|I_0|.
\end{equation}
As a consequence, for $K>0$ sufficiently large, we conclude that, at every generation between $0$ and $L$, the total length of the selected intervals is bounded below by a fixed positive multiple of $|I_0|$. Since the intervals at each generation form a partition of $I_0$, their total length is also bounded above by $|I_0|$. Hence
\begin{equation}\label{20260828eq11}
\sum_{k=0}^{L-1}
\sum_{\substack{I\in\mathcal S_L(I_0)\\
I\ {\rm is\ selected}\\|I|=2^{-k}|I_0|}}
|I|
\simeq L|I_0|.
\end{equation}

We now prove \eqref{20260828eq01}. For each $0\leq k\leq L$, the intervals $I\in\mathcal S_L(I_0)$ satisfying $|I|=2^{-k}|I_0|$ form a partition of $I_0$. It follows from \eqref{20260828eq09} that
\begin{align*}
\sum_{I\in\mathcal S_L(I_0)}D(I)^{\frac{p}{p-2}}|I|
&\simeq_p
a^{\frac{p}{p-2}}
\sum_{k=0}^L
\sum_{\substack{I\in\mathcal S_L(I_0)\\|I|=2^{-k}|I_0|}}
|I|\\
&=(L+1)a^{\frac{p}{p-2}}|I_0|\\
&\simeq_p N^{-2}a|I_0|,
\end{align*}
where the last estimate follows from \eqref{20260828eq05}. This proves \eqref{20260828eq01}.

\vspace{0.1cm}

To prove \eqref{20260828eq02}, observe from \eqref{20260828eq06} and \eqref{20260828eq07} that
$$
|D(I_+)-D(I)|
=
\begin{cases}
\eta,&I\text{ is selected},\\
\\
0,&I\text{ is unselected}.
\end{cases}
$$
Therefore, by \eqref{20260828eq11},
\begin{align*}
\sum_{\substack{I\in\mathcal S_L(I_0)\\|I|>2^{-L}|I_0|}}
|D(I_+)-D(I)|^{\frac{p}{p-1}}|I|
&=
\eta^{\frac{p}{p-1}}
\sum_{k=0}^{L-1}
\sum_{\substack{I\in\mathcal S_L(I_0)\\
I\ {\rm is\ selected}\\|I|=2^{-k}|I_0|}} |I|\\
&\qquad\simeq_p
L\eta^{\frac{p}{p-1}}|I_0|\\
&\qquad=
\frac{1}{K^{\frac{p}{p-1}}} \cdot 
a^{\frac{p}{p-1}}
L^{\frac{p-2}{2(p-1)}}|I_0|\\
&\qquad\simeq_p
N^{-\frac{p-2}{p-1}}a|I_0|,
\end{align*}
where we again used \eqref{20260828eq05}. This proves \eqref{20260828eq02}.

It remains to prove the localized estimates. Fix $J\in\mathcal S_L(I_0)$. By \eqref{20260828eq09} again, we have
\begin{align*}
\sum_{\substack{I\in\mathcal S_L(I_0)\\I\subseteq J}}
D(I)^{\frac{p}{p-2}}|I| \lesssim_p(L+1)a^{\frac{p}{p-2}}|J| \lesssim_p N^{-2}a|J| \lesssim_p N^{-2}D(J)|J|.
\end{align*}
This proves \eqref{20260828eq03}. Similarly,
\begin{align*}
\sum_{\substack{I\in\mathcal S_L(I_0)\\I\subseteq J\\|I|>2^{-L}|I_0|}}
|D(I_+)-D(I)|^{\frac{p}{p-1}}|I|
\leq
L\eta^{\frac{p}{p-1}}|J| \lesssim_p
N^{-\frac{p-2}{p-1}}a|J| \lesssim_p
N^{-\frac{p-2}{p-1}}D(J)|J|.
\end{align*}
This proves \eqref{20260828eq04} and completes the proof.
\end{proof}

We are now ready to prove Theorem~\ref{mainthm02}. The main idea is to obtain the desired measures by iterating the construction in Lemma~\ref{20260828lem01}.

\begin{proof}[Proof of Theorem~\ref{mainthm02}]
Fix an integer $N\geq2$. Choose $a_0>0$ sufficiently small, set
$$
\mathcal R_0:=\{\T\},\qquad D(\T):=a_0,
$$
and apply Lemma~\ref{20260828lem01} with $I_0=\T$ and $a=a_0$. Denote the resulting integer by $L_\T$, and let $\mathcal R_1$ consist of the intervals $I\in\mathcal S_{L_\T}(\T)$ satisfying
$$
|I|=2^{-L_\T}|\T|.
$$
Thus, $\mathcal R_1$ is the $L_\T$-th dyadic generation below $\T$. Suppose that $\mathcal R_{m-1}$ has been constructed. For every $R\in\mathcal R_{m-1}$, apply Lemma~\ref{20260828lem01} with $I_0=R$ and $a=D(R)$, and denote the resulting integer by $L_R$. Define
$$
\mathcal R_m:=\bigcup_{R\in\mathcal R_{m-1}}\left\{I\in\mathcal S_{L_R}(R):|I|=2^{-L_R}|R|\right\}.
$$
Continue this construction for $N^2$ steps. By \eqref{20260828eq09}, an induction on $m$ gives
$$
D(R)\leq 2^m a_0\leq 2^{N^2}a_0,\qquad R\in\calR_m,\quad 0\leq m\leq N^2.
$$
Therefore, we may choose $a_0>0$ so small that $2^{N^2}a_0$ is sufficiently small for Lemma~\ref{20260828lem01} to apply at every step of the construction.

For every $0\leq m\leq N^2$, the intervals in $\mathcal R_m$ are pairwise disjoint and form a partition of $\T$. Let
$$
\mathcal G:=\bigcup_{m=1}^{N^2}\bigcup_{R\in\mathcal R_{m-1}}\mathcal S_{L_R}(R)
$$
be the finite collection of all intervals arising in the construction. By the above construction, \eqref{20260828eq08} holds at every $I\in\mathcal G\setminus\mathcal R_{N^2}$, and since $|I_+|=|I_-|=|I|/2$, it follows that
\begin{equation}\label{20260828eq12}
D(I)|I|=D(I_+)|I_+|+D(I_-)|I_-|,\qquad I\in\mathcal G\setminus\mathcal R_{N^2}.
\end{equation}

For each $R\in\mathcal R_{N^2}$, choose a point $z_R\in Q_R^{\textrm{up}}$ and define
\begin{equation}\label{20260828eq13}
\mu^{(N)}:=\sum_{R\in\mathcal R_{N^2}}D(R)|R| \cdot \one_{z_R}.
\end{equation}
We claim that
\begin{equation}\label{20260828eq14}
\mu^{(N)}(Q_I)=D(I)|I|,\qquad I\in\mathcal G.
\end{equation}
Indeed, \eqref{20260828eq14} follows directly from \eqref{20260828eq13} when $I\in\mathcal R_{N^2}$. Suppose that it holds for the two children of some $I\in\mathcal G\setminus\mathcal R_{N^2}$. Then \eqref{20260828eq12} gives
\begin{align*}
\mu^{(N)}(Q_I)
&=\mu^{(N)}(Q_{I_+})+\mu^{(N)}(Q_{I_-})\\
&=D(I_+)|I_+|+D(I_-)|I_-|\\
&=D(I)|I|.
\end{align*}
Thus, \eqref{20260828eq14} follows by backward induction. In particular, this means that $\mu^{(N)}(\D)=D(\T)|\T|<+\infty$, and hence $\mu^{(N)}$ is a finite measure. 

We also claim that
\begin{equation}\label{20260828eq15}
\mu^{(N)}(Q_I)=0,\qquad I\in\calD\setminus\mathcal G.
\end{equation}
Indeed, the intervals in $\mathcal R_{N^2}$ form a partition of $\T$, and $\mathcal G$ contains every dyadic ancestor of each interval in $\mathcal R_{N^2}$. Thus, every $I\in\calD\setminus\mathcal G$ is strictly contained in a unique interval $R\in\mathcal R_{N^2}$. Since $z_R\in Q_R^{\textrm{up}}$, the box $Q_I$ contains none of the points supporting $\mu^{(N)}$. This proves \eqref{20260828eq15}.

\medskip 

It remains to estimate the packing and Haar energies of $\mu^{(N)}$.

\vspace{0.1cm}

\noindent{\textbf{Estimate of the packing energy of $\mu^{(N)}$.}}
Fix $J\in\mathcal G$. Let $m_0$ be the smallest integer for which there exists $R_0\in\mathcal R_{m_0-1}$ such that
$J\in\mathcal S_{L_{R_0}}(R_0)$. By \eqref{20260828eq03}, the contribution from the part of this block lying below $J$ satisfies
\begin{equation}\label{20260829eq01}
\sum_{\substack{I\in\mathcal S_{L_{R_0}}(R_0)\\I\subseteq J}}D(I)^{\frac{p}{p-2}}|I|\lesssim_p N^{-2}D(J)|J|.
\end{equation}
For each $m_0<m\leq N^2$, the intervals $R\in\mathcal R_{m-1}$ contained in $J$ form a pairwise disjoint partition of $J$. Applying \eqref{20260828eq12} repeatedly, we have 
\begin{equation}\label{20260829eq02}
\sum_{\substack{R\in\mathcal R_{m-1}\\R\subseteq J}}D(R)|R|=D(J)|J|.
\end{equation}
Applying \eqref{20260828eq03} to the block rooted at each such $R$ and then using \eqref{20260829eq02}, we obtain
\begin{align}
\sum_{\substack{R\in\mathcal R_{m-1}\\R\subseteq J}}\sum_{I\in\mathcal S_{L_R}(R)}D(I)^{\frac{p}{p-2}}|I|
&\lesssim_p N^{-2}\sum_{\substack{R\in\mathcal R_{m-1}\\R\subseteq J}}D(R)|R|\nonumber\\
&=N^{-2}D(J)|J|.
\label{20260829eq03}
\end{align}
The intervals in $\mathcal G$ contained in $J$ are covered by the part of the first block appearing in \eqref{20260829eq01} and the subsequent blocks appearing in \eqref{20260829eq03}. Hence, by \eqref{20260828eq14}, \eqref{20260828eq15}, \eqref{20260829eq01}, and \eqref{20260829eq03}, we have 
\begin{align*}
&\sum_{\substack{I\in\calD\\I\subseteq J}}\left(\frac{\mu^{(N)}(Q_I)}{|I|}\right)^{\frac{p}{p-2}}|I| =\sum_{\substack{I\in\mathcal G\\I\subseteq J}}D(I)^{\frac{p}{p-2}}|I| \\
&\leq\sum_{\substack{I\in\mathcal S_{L_{R_0}}(R_0)\\I\subseteq J}}D(I)^{\frac{p}{p-2}}|I|+\sum_{m=m_0+1}^{N^2}\sum_{\substack{R\in\mathcal R_{m-1}\\R\subseteq J}}\sum_{I\in\mathcal S_{L_R}(R)}D(I)^{\frac{p}{p-2}}|I|\\
&\lesssim_p (N^2-m_0+1)N^{-2}D(J)|J|\\
&\lesssim_p D(J)|J|=\mu^{(N)}(Q_J).
\end{align*}
By \eqref{20260828eq15}, every $J\in\calD$ satisfying $\mu^{(N)}(Q_J)>0$ belongs to $\mathcal G$. Taking the supremum over all such $J$, we conclude that
\begin{equation}\label{20260828eq16}
\calC_{p,\calD}(\mu^{(N)})\lesssim_p1.
\end{equation}
\medskip

\noindent{\textbf{Estimate of the Haar energy of $\mu^{(N)}$.}}
For every $I\in\mathcal G\setminus\mathcal R_{N^2}$, \eqref{20260828eq08} and \eqref{20260828eq14} give
\begin{align*}
\mu^{(N)}(Q_{I_+})-\mu^{(N)}(Q_{I_-})
&=\frac{|I|}{2}\bigl(D(I_+)-D(I_-)\bigr)\\
&=|I|\bigl(D(I_+)-D(I)\bigr).
\end{align*}
Therefore,
\begin{equation}\label{20260828eq17}
\left(\frac{\left|\mu^{(N)}(Q_{I_+})-\mu^{(N)}(Q_{I_-})\right|}{|I|}\right)^{\frac{p}{p-1}}|I|
=|D(I_+)-D(I)|^{\frac{p}{p-1}}|I|.
\end{equation}

Fix $1\leq m\leq N^2$. Using \eqref{20260828eq02} with $I_0$ replaced by each $R\in\mathcal R_{m-1}$, and then iterating \eqref{20260828eq12}, we obtain
\begin{align} \label{20260829eq53}
\sum_{R\in\mathcal R_{m-1}}\sum_{\substack{I\in\mathcal S_{L_R}(R)\\|I|>2^{-L_R}|R|}}
|D(I_+)-D(I)|^{\frac{p}{p-1}}|I|
&\gtrsim_p N^{-\frac{p-2}{p-1}}\sum_{R\in\mathcal R_{m-1}}D(R)|R| \nonumber \\
&=N^{-\frac{p-2}{p-1}}D(\T)|\T|.
\end{align}
By the definition of $\mathcal R_m$, the collections
$$
\bigcup_{R\in\mathcal R_{m-1}}\left\{I\in\mathcal S_{L_R}(R):|I|>2^{-L_R}|R|\right\},
\qquad 1\leq m\leq N^2,
$$
are pairwise disjoint. This together with \eqref{20260828eq17} and \eqref{20260829eq53} gives
\begin{align*}
&\sum_{\substack{I\in\calD\\I\subseteq\T}}
\left(\frac{\left|\mu^{(N)}(Q_{I_+})-\mu^{(N)}(Q_{I_-})\right|}{|I|}\right)^{\frac{p}{p-1}}|I|\\
&\geq\sum_{m=1}^{N^2}\sum_{R\in\mathcal R_{m-1}}\sum_{\substack{I\in\mathcal S_{L_R}(R)\\|I|>2^{-L_R}|R|}}
\left(\frac{\left|\mu^{(N)}(Q_{I_+})-\mu^{(N)}(Q_{I_-})\right|}{|I|}\right)^{\frac{p}{p-1}}|I|\\
&=\sum_{m=1}^{N^2}\sum_{R\in\mathcal R_{m-1}}\sum_{\substack{I\in\mathcal S_{L_R}(R)\\|I|>2^{-L_R}|R|}}
|D(I_+)-D(I)|^{\frac{p}{p-1}}|I|\\
&\gtrsim_p\sum_{m=1}^{N^2}N^{-\frac{p-2}{p-1}}D(\T)|\T|=N^{\frac{p}{p-1}}D(\T)|\T|.
\end{align*}
On the other hand, \eqref{20260828eq14} gives
$\mu^{(N)}(Q_\T)=D(\T)|\T|$. It follows from the definition of the Haar energy that
\begin{align*}
\calH_{p,\calD}(\mu^{(N)})
&\geq\frac{1}{\mu^{(N)}(Q_\T)}
\sum_{\substack{I\in\calD\\I\subseteq\T}}
\left(\frac{\left|\mu^{(N)}(Q_{I_+})-\mu^{(N)}(Q_{I_-})\right|}{|I|}\right)^{\frac{p}{p-1}}|I| \gtrsim_p N^{\frac{p}{p-1}}.
\end{align*}
Together with \eqref{20260828eq16}, this completes the proof of Theorem~\ref{mainthm02}.
\end{proof}

\end{document}